\documentclass[11pt,draftcls,onecolumn]{IEEEtran}

\usepackage[utf8x]{inputenc}
\usepackage{graphicx,epsfig,epstopdf}
\usepackage{stmaryrd}
\usepackage{amssymb}
\usepackage{amsfonts}
\usepackage{latexsym}
\usepackage{amsmath}
\usepackage{mathrsfs}
\usepackage{bbm}
\usepackage{amsthm}
\usepackage{color}
\usepackage{cite}
\usepackage{subfig}
\usepackage[font=footnotesize]{caption}
\usepackage{bm}
\usepackage{mathrsfs} 
\usepackage{subfig}
\usepackage{comment}

\includecomment{longversion} 
\excludecomment{shortversion} 

\includecomment{onecol}
\excludecomment{twocol}

\usepackage[T1]{fontenc}

\usepackage{algorithm}
\usepackage{algorithmic}

\newcommand{\myinclude}[1]{\includegraphics[width=3.2cm, height=3cm]{#1}}
\newcommand{\algo}{CENTREx}
\newcommand{\dalgo}{DeCENTREx}

\newcommand{\term}{\Sigma}

\newcommand{\pastorv}[1]{\textcolor{black}{#1}}

\newcommand{\duprazv}[1]{\textcolor{black}{#1}}
\newcommand{\mycdot}{\bullet}
\newcommand{\NewStd}{r}
\newcommand{\TheThreshold}[1]{#1\mu(\level)}
\newcommand{\Topt}[1]{\test_{\TheThreshold{#1}}}

\newcommand{\TheFunc}{h_N}
\def\myvspace{\vspace{0cm}}

\def\mycdot{\bullet}
\def\d{\mathrm{d}}
\newcommand{\VZ}[1]{\pastorv{\Zbm(#1)}}

\def\dev{\Deltabm}
\def\myvspace{\vspace{0.2cm}}
\newcommand{\myQ}[1]{\Marcum \left ( \tol , #1 \right )}

\newcommand{\estCtr}{\widehat{\Ctr}}
\newcommand{\dt}{{\mathrm d} t}

\newcommand{\dz}{{\mathrm d} z}

\def\Ibf{\mathbf{I}}

\newcommand{\Bbm}{{\bm B}}

\newcommand{\Psibm}{{\bm \Psi}}
\newcommand{\Deltabm}{{\bm \Delta}}
\newcommand{\thetabm}{{\bm \theta}}

\newcommand{\Tbm}{{\bm T}}

\newcommand{\xibm}{{\bm \xi}}
\newcommand{\Xbm}{{\bm X}}

\newcommand{\tbm}{{\bm t}}
\newcommand{\xbm}{{\bm x}}
\newcommand{\Ybm}{{\bm Y}}
\newcommand{\ybm}{{\bm y}}
\newcommand{\zbm}{{\bm z}}

\newcommand{\Zbm}{{\bm Z}}

\newcommand{\fbm}{{\bm f}}

\def\Ebb{\mathbb{E}}
\def\Fbb{\mathbb{F}}
\def\Gbb{\mathbb{G}}

\def\Pbb{\mathbb{P}}

\def\Rbb{\mathbb{R}}

\def\Hcal{\mathcal{H}}

\def\Ncal{\mathcal{N}}

\def\Scal{\mathcal{Y}}

\newcommand{\Tfrak}{\mathfrak T}

\def\ANcal{\mathcal{AN}}

\def\Rset{\Rbb}

\def\dim{d}

\newcommand{\Expect}[1]{\Ebb \left [ \, #1 \, \right ]}
\newcommand{\tol}{\tau}

\newcommand{\transpose}{\mathrm{T}}

\newcommand{\Noise}{\Xbm}

\newcommand{\obs}{\Ybm}
\newcommand{\Obs}{\Ybm}

\newcommand{\mybig}{\big}

\newcommand{\test}{\Tfrak}

\def\radius{\rho}
\def\level{\gamma}
\def\tsub{\lambda_\level}
\def\tsubprime{\lambda_{\level'}}

\newcommand{\Proba}[1]{\Pbb \left [ \, #1 \, \right ]}

\newcommand{\Identity}[1]{\Ibf_{#1}}

\def\std{\sigma}

\newcommand{\Ctr}{\thetabm}
\newcommand{\ctr}{\theta}

\newcommand{\Marcum}{Q_{\dim/2}}
\newcommand{\Id}{\Identity{\dim}}
\newcommand{\pval}[2]{\widehat{\level}_{#2}(#1)}
\newcommand{\thepval}[1]{\widehat{\level}_{\sqrt{2}}(#1)}
\newcommand{\IfF}{\text{IF}}
\newcommand{\Mk}{\mathcal{M}}
\newcommand{\Ectr}{\Phi}
\newcommand{\Ck}{\mathcal{C}_k}

\newtheorem{Proposition}{Proposition}
\newtheorem{Lemma}{Lemma}

\title{Decentralized Clustering based on Robust Estimation and Hypothesis Testing}
\author{Dominique Pastor, Elsa Dupraz, François-Xavier Socheleau \\ \small IMT Atlantique, Lab-STICC, Univ. Bretagne Loire, Brest, France}

\begin{document}

\maketitle

\begin{abstract}
This paper considers a network of sensors without fusion center that may be difficult to set up in applications involving sensors embedded on autonomous drones or robots. 
In this context, this paper considers that the sensors must perform a given clustering task in a fully decentralized setup.
Standard clustering algorithms usually need to know the number of clusters and are very sensitive to initialization, which makes them difficult to use in a fully decentralized setup. 
In this respect, this paper proposes a decentralized model-based clustering algorithm that overcomes these issues. 
The proposed algorithm is based on a novel theoretical framework that relies on hypothesis testing and robust M-estimation. 
More particularly, the problem of deciding whether two data belong to the same cluster can be optimally solved via Wald's hypothesis test on the mean of a Gaussian random vector. The p-value of this test makes it possible to define a new type of score function, particularly suitable for devising an M-estimation of the centroids. The resulting decentralized algorithm efficiently performs clustering without prior knowledge of the number of clusters. It also turns out to be less sensitive to initialization than the already existing clustering algorithms, which makes it appropriate for use in a network of sensors without fusion center.
\end{abstract}

\section{Introduction}\label{sec:intro}
Networks of sensors are now used in a wide range of applications in medicine, in telecommunications, or in environmental domains~\cite{yick08CN}.
They are employed, for example, for human health monitoring~\cite{omeni08BCS}, activity recognition on home environments~\cite{ordonez13S}, spectrum sensing in cognitive radio~\cite{sahasranand15ICC}, and so forth. 
In these applications, a fusion center can collect all the data from all the sensors and perform a given estimation or learning task over the collected data. 
However, it is not always practical to set up a fusion center, especially in recent applications involving autonomous drones or robots~\cite{tuna14AHN}.
In such applications in which no fusion center is available, the sensors should perform the learning task by themselves in a fully decentralized setup.

In this paper, we assume that the sensors have to perform decentralized clustering on the data measured within the network.
The aim of clustering is to divide the data into clusters such that the data inside a cluster are similar with each other and different from the data that belong to other clusters.
Clustering is considered in various applications of sensor networks, such as parking map construction~\cite{zhou16ICC} or controller placement in telecommunication networks~\cite{wang16ICC}.
One of the most popular clustering algorithms is K-means~\cite{jain10PRL}, due to its simplicity and its efficiency. 

The K-means algorithm groups $N$ measurement vectors into $K$ clusters with a two-step iterative procedure. 
It was proved that this iterative procedure always converges to a local minimum~\cite{jain10PRL}. However, in order to get a chance to reach the global minimum, the K-means algorithm needs to be initialized properly. 
Proper initialization can be obtained with the K-means++ procedure~\cite{arthur07SIAM}, which requires computing all the two by two distances between all the measurement vectors in the dataset. 
As another issue, the K-means algorithm need to know the number $K$ of clusters. When $K$ unknown, it is possible to apply a penalized method that requires applying the K-means algorithm several times with different numbers of cluster~\cite{pelleg2000x}.
It is worth mentioning that the variants of K-means such as Fuzzy K-means~\cite{wu02PR} suffer from the same two issues.

The K-means algorithm was initially introduced for non-distributed setups, but decentralized versions of the algorithm have also been proposed~\cite{datta09KDE,di11DMW,fellus13dDMW}. 
In a decentralized setup, each of the $N$ sensors initially observes one single vector that correspond to its own measurements.
Then, in order to apply the decentralized K-means algorithm, the sensors are allowed to exchange some data with each other. 
The objective of decentralized clustering is thus to allow each sensor to perform the clustering with only partial observations of the available data, while minimizing the amount of data exchanged in the network. 
As a result, in this context, it is not desirable to initialize the algorithm with the K-means++ procedure that would require exchanging all the two by two distances between all the measurement vectors. 
It is not desirable either to perform the distributed algorithm several times in order to determine the number of clusters. 

The above limitations have not been addressed in the previous works~\cite{datta09KDE,di11DMW,fellus13dDMW}, and the objective of this paper is thus to propose a decentralized clustering algorithm that overcomes them.
At first sight, the algorithms DB-SCAN~\cite{ester1996density} and OPTICS~\cite{ankerst1999optics} may appear as good candidates for decentralized clustering since they do need the number of clusters and since they do not have any initialization issues. However, they require setting two parameters that are the maximum distance between two points in a cluster and the minimum number of points per cluster. These parameters can hardly be estimated and they must be chosen empirically, which we would like to avoid. This is why we do not consider these solutions here. 

The K-means algorithm makes no assumption on the signal model of the measurement vectors that belong to a cluster, which is relevant for applications such as document classification~\cite{steinbach2000comparison} information retrieval, or categorical data clustering~\cite{huang1998extensions}. 
On the other hand, signal processing methods usually assume a statistical model on the measurements. This model can be derived, for example, from physical characteristics of the sensors.
In this respect, centralized and decentralized model-based clustering algorithms were proposed in~\cite{bouveyron14CSDA,forero11SP}, although they suffer from the same two issues as the K-means algorithms.
In the model-based clustering {algorithms} proposed in~\cite{bouveyron14CSDA,forero11SP}, the measurement vectors that belong to a given cluster are modeled as the cluster centroid plus Gaussian noise, and it is assumed that both the cluster centroid and the noise variance are unknown. 
However, the noise variance may be estimated, possibly from preliminary measurements, via a bunch of parametric, non-parametric, and robust methods (see \cite{Huber2009, Rousseeuw93, Pastor2012a}, among others).
Therefore, here, we will consider the same Gaussian model as in~\cite{bouveyron14CSDA, forero11SP}, but we will assume that the noise variance is known.
This assumption was already made for clustering in~\cite{zhou16ICC} and~\cite{wu02PR} in order to choose the parameters for the functions that compute the cluster centroids.

In what follows, under the assumption that the noise variance is known, we propose a novel clustering algorithm which does not require prior knowledge of the number of clusters and which is much less sensitive to initialization than the K-means algorithm.
The centralized and decentralized versions of the algorithm we propose are both based on the robust estimation of the cluster centroids and on the testing of whether a given measurement vector belongs to a given cluster. Whence the names \algo~--- for CENtroids Testing and Robust Estimation --- and \dalgo, respectively given to the centralized and the decentralized algorithm.

In both algorithms, the cluster centroids are estimated one after the other via robust M-estimation~\cite{zoubir12robust}, assuming that the measurement vectors from the other clusters are outliers. 
In order to estimate the centroids, M-estimation looks for the fixed points of a function whose expression depend on a score function applied to all the measurement vectors of the database.
The score function we choose is the p-value of the Wald {hypothesis} test for testing the mean of a Gaussian~\cite{Wald1943} and evaluates the plausibility that a measurement vector belongs to a given cluster.
M-estimation was already used in~\cite{wu02PR} to estimate the cluster centroids, with a different score function. 
In~\cite{wu02PR}, the robustness of the centroid estimation was evaluated from the standard M-estimation approach, which shows that an outlier of infinite amplitude gives only a finite estimation error. 
Here, we propose an alternative analysis that, unlike \cite{wu02PR}, takes into account the fact that the outliers are measurement vectors from other clusters. 
Our asymptotic analysis shows that the only fixed points of our M-estimation function are the true cluster centroids, which validates our approach. 
We also derive the statistics of the estimation error for a finite number of measurement vectors.

In our algorithm, for each centroid to be estimated, the iterative computation of one of the fixed points of the M-estimation function is simply initialized with one of the measurement vectors of the dataset. 
The iterative computation then retrieves the centroid of the cluster to which the initialization point belongs. 
After each centroid estimation, a Wald hypothesis test is applied to mark all the measurement vectors that must not be used later for initializing the estimation of any other centroid because they are already close enough to the newly estimated one.
This very simple {marking operation} avoids using the K-means++ solution. Further, the estimation process stops when all the measurement vectors have been marked, which permits to determine the number of clusters.
The final clustering is standardly performed by seeking the estimated centroid that is the closest to a given observation.
Our simulation results show that both \algo~and \dalgo~achieve performance close to the K-means algorithm initialized with the proper number of clusters. 

The outline of the paper is as follows. Section~\ref{sec:model} describes the signal model we consider for the measurement vectors.
Section~\ref{sec:theory} details the theoretical framework, which involves the derivation of all the functions and hypothesis tests used in the successive steps of the algorithm.
Section~\ref{sec:algo} describes the centralized and decentralized versions of the clustering algorithm we propose. 
Section~\ref{sec:experiments} shows the simulation results and Section~\ref{sec:conclusion} gives our conclusions and perspectives for future works.

\section{Signal model}\label{sec:model}
This section describes the signal model and the notation used throughout the paper. 
Consider a network of $N$ sensors in which each sensor $n \in \{1,\cdots,N\}$ observes a measurement vector $\obs_n$. 
The vectors $\obs_1, \ldots, \obs_N$ are assumed to be
$N$ independent $\dim$-dimensional random Gaussian vectors. 
The individual random components of each random vector $\Obs_n$ are assumed to be independent and identically distributed (i.i.d.).
Note that this i.i.d. simplyfing assumption is considered here as a fist step to introduce the analysis and more accurate models will be considered in future works {(for example, observations with different variances, see conclusion for more details)}. 
To alleviate notation in upcoming computations, we conveniently assume that the covariance matrices of the observations $\Obs_n$ are normalized so as to all equal the $\dim \times \dim$ identity matrix $\Id$. 
This normalization requires prior knowledge of the noise variance, which may be known or estimated by various parametric and nonparametric methods,
as mentioned in the introduction.

We further assume that the measurement vectors are split into $K$ clusters defined by $K$ centroids $\Ctr_1, \ldots, \Ctr_K$, with $\Ctr_k \in \Rset^\dim$ for each $k = 1, 2, \ldots, K$. 
Therefore, according to the foregoing assumption on the noise measurement, we assume that for each $n \in \{ 1, \ldots, N \}$, there exists $k \in \{1, 2, \ldots, K\}$ such that $\Obs_n \thicksim \Ncal(\Ctr_k, \Id)$ and we say that $\obs_n$ belongs to cluster $k$. 
For each $k$, we denote by $N_k$ the number of measurement vectors in cluster $k$. We also use the notation $\Obs_{k,1},\Obs_{k,2}, \ldots, \Obs_{k,N_k}$ to designate the $N_k$ observations $\Obs_n$ that belong to a given cluster $k$.
In the following, we assume that the number of clusters $K$, the centroids $\Ctr_1, \ldots, \Ctr_K$, and the number of measurement vectors in each cluster $N_1,\cdots, N_K$ are unknown. 
The objective of the clustering algorithm we propose is to estimate these quantities and also to determine to which cluster each measurement vector $\obs_n$ belongs.

\section{Theoretical framework}\label{sec:theory}
In our algorithm, the centroids are estimated one after each other by a robust M-estimation approach. As described in~\cite{zoubir12robust} {and references therein}, the M-estimation involves searching the fixed points of a given function.
When the algorithm has estimated a new centroid, it identifies and marks the observations $\obs_n$ that are too close to the newly estimated cluster to be used later for estimating other centroids. 
For this, we apply a Wald hypothesis test~\cite{Wald1943} that decides whether a given vector must be marked or not.

In this section, we introduce the theoretical framework on which rely the different parts of the algorithm. In particular, we present the M-estimation of the centroids and the hypothesis test used to build up the clusters.
As described in the following, the function we use for M-estimation is deduced from Wald's test. We further analytically show that the centroids are the only fixed-points of this function, which justifies that our algorithm can successfully recover the clusters. 

\subsection{The Wald test and its p-value for testing the mean of a Gaussian} 
\label{Subsec: Wald}

{Wald tests are optimal in a specific sense defined in~\cite[Definition III, p. 450]{Wald1943} to test the mean of a Gaussian \cite[Proposition III, p. 450]{Wald1943}. With respect to the model assumptions introduced in Section~\ref{sec:model}, Wald tests are hereafter exposed where the Gaussian has identity scale covariance matrix. In the sequel, Wald tests will serve: (i) to define the function we use for M-estimation of the centroids, (ii) to mark the measurement vectors that are close to the estimated centroids. 
Here, we first describe the test in a generic way and we will apply it later in the section for the two purposes recalled above.

Let $\Zbm$ be a real $\dim$-dimensional random vector such that $\Zbm \thicksim \Ncal(\xibm,\sigma_0^2 \Id)$ with $\sigma_0 \neq 0$. Consider the problem of testing the mean of the Gaussian vector $\Zbm$, that is the problem of testing the null hypothesis $\Hcal_0: \xibm = 0$ against its alternative $\Hcal_1: \xibm \neq 0$. This problem is summarized as:
\begin{equation}
\label{Eq:ddtpb}
\left \{
\begin{array}{lll}
\text{{Observation:}} \, \Zbm \thicksim \Ncal(\xibm,\sigma_0^2 \Id), \\
\text{Hypotheses:} \,
\left \{
\begin{array}{lll}
\Hcal_0: \, \xibm= 0, \\
\Hcal_1: \, \xibm \neq 0.
\end{array}
\right.
\end{array}
\right.
\end{equation}

Recall that a non-randomized test is any (measurable) map of $\Rset^\dim$ to $\{0,1\}$ and that, given some test $\test$ and $\zbm \in \Rset^\dim$, the value $\test(\zbm)$ is the index of the hypothesis accepted by $\test$ at $\zbm$. For instance, if $\test(\zbm) = 0$ (resp. $\test(\zbm) = 1$), $\test$ accepts $\Hcal_1$ (resp. $\Hcal_0$) when $\zbm$ is the realization of $\Zbm$. We can then devise easily a non-randomized test that guarantees a given false alarm probability $\level \in (0,1)$ for testing $\Hcal_0$ against $\Hcal_1$. Indeed, for any given $\lambda \in [0,\infty)$, let $\test_\lambda$ be the test defined for any $\zbm \in \Rset^\dim$ by setting:
\begin{equation}
\label{Eq:Thresholding test from above}
\test_\lambda(\zbm) = \left \{
\begin{array}{lll}
0 & \hbox{ if } & \| \zbm \| \leqslant \lambda \\
1 & \hbox{ if } & \| \zbm \| > \lambda.
\end{array}
\right.
\end{equation}
where $\| \mycdot \|$ denotes the usual Euclidean norm in $\Rset^\dim$. This test accepts $\Hcal_0$ (resp. $\Hcal_1$) if $\| \zbm \| \leqslant \lambda$ (resp. $\| \zbm \| > \lambda$). According to the definition of the Generalized Marcum Function $\Marcum(\mycdot,\mycdot)$ \cite[Eq. (8)]{Sun2010}, the false alarm probability of this test is $\Proba{\| \Zbm \|^2 > \lambda^2} = 1 - \Fbb_{\chi^2_d(0)}(\lambda^2/\sigma_0^2)  = \Marcum(0,\lambda/\sigma_0)$ where $\Zbm \thicksim \Ncal(0,\sigma_0^2 \Id)$ and $\Fbb_{\chi^2_\dim(0)}$ is the cumulative distribution function (CDF) of the centered $\chi^2_\dim$ distribution with $\dim$ degrees of freedom. {Throughout {the paper}, we denote by $\mu(\level)$ the unique real value such that:
\begin{equation}
\label{Eq: normalized threshold}
\Marcum(0,\mu(\level)) = \level
\end{equation}}
It then follows that the false alarm probability of the test $\Topt{\sigma_0}$ equates the desired value $\level$. 

Although there is no Uniformly Most Powerful (UMP) test for the composite binary hypothesis testing problem (\ref{Eq:ddtpb}) \cite[Sec. 3.7]{Lehmann}, $\Topt{\sigma_0}$ turns out to be optimal with respect to several optimality criteria and within several classes of tests with level $\level$ \cite[Proposition 2]{RDT}. In particular, $\Topt{\sigma_0}$ is UMP among all spherically invariant tests since it has Uniformly Best Constant Power (UBCP) on the spheres centered at the origin of $\Rset^\dim$ \cite[Definition III \& Proposition III, p. 450]{Wald1943}. In the sequel, any test $\Topt{\sigma_0}$ will be called a Wald test, without recalling explicitly the level $\gamma$ at which the testing is performed. 


It turns out that a notion of p-value can be defined for the family of Wald tests $\test_{\sigma_0 \mu(\level)}$ when $\level$ ranges in $(0,1)$. This p-value is calculated in Appendix \ref{Sec: RDT pvalue} and, for the testing problem (\ref{Eq:ddtpb}), it is given for any $\zbm \in \Rset^\dim$ by:
\begin{equation}
\label{RDT-pvalue}
\pval{\zbm}{\sigma_0} = \Marcum(0, \| \zbm \| / \sigma_0).
\end{equation}
The p-value can be seen as a measure of the plausibility of the null hypothesis $\Hcal_0$. In particular, when $\| \zbm \|$ tends to $+\infty$, $\pval{\zbm}{\sigma_0}$ tends to $0$ since $\displaystyle \lim_{t \to \infty} \Marcum(0,t) = \displaystyle \lim_{t \to \infty} ( 1 - \Fbb_{\chi^2_d(0)}(t^2)) = 0$. It is then natural to consider that the plausibility of $\Hcal_0$ vanishes as $\| \zbm \|$ grows to $+\infty$. Similarly, $\pval{\zbm}{\sigma_0}$ tends to $1$ when $\| \zbm \|$ tends to $0$, so that the plausibility of $\Hcal_0$ is rather high, close to $1$, for small values of $\| \zbm \|$. Accordingly, we describe the M-estimation of the centroids and show how the Wald test and its p-value help us choose the score function that will be used in the $M$ estimation.

\vspace{-0.3cm}

\subsection{M-estimation of the centroids}
\label{sec:mestimator}

As in \cite{wu02PR}, we want to estimate the centroid $\Ctr_k$ of a given cluster with an M-estimator, which amounts to considering that the measurement vectors from other clusters are outliers. More specifically, if the number $K$ of clusters were known, robust estimation theory \cite{Hampel74, Hampel86, Vaart1998, Huber2009, zoubir12robust} applied to the problem of estimating the centroids $\Ctr_1,\cdots,\Ctr_K$ would lead to calculating the solutions $\estCtr_1, \ldots, \estCtr_K$ of the $K$ equations:
\begin{equation}
\label{Eq:argmin rho}
(\estCtr_1, \ldots, \estCtr_K) = \displaystyle \underset{(\Ctr_1, \ldots, \Ctr_K)}{\arg} \min J(\Ctr_1,\ldots,\Ctr_K),
\vspace{-0.25cm}
\end{equation}
where $J(\Ctr_1,\ldots,\Ctr_K) = \sum_{k=1}^K \sum_{n = 1}^{N_k} \rho \left ( \| \obs_{k,n}-\Ctr_k \|^2 \right )$
and the loss function $\rho: \Rset \rightarrow \Rset$ is an increasing function. 
If the loss function $\rho$ is differentiable, the solution to (\ref{Eq:argmin rho}) can be found by solving the $K$ equations
\begin{align}
\partial_k J(\Ctr_1, \ldots, \Ctr_K) 
& = \displaystyle \sum_{n = 1}^{N_k} (\obs_{k,n}-\Ctr_k) w ( \| \obs_{k,n}-\Ctr_k \|^2 ) \notag \\
\label{Eq:psi = 0}
& = 2 \displaystyle \sum_{n = 1}^{N_k} \Psi(\obs_{k,n}-\Ctr_k) = 0
\end{align}
for $k = 1, \ldots, K$, where $\partial_k J$ is the partial derivate of $J$ with respect to its $k$th argument, $w = \rho'$ and $\Psi: \Rset^\dim \rightarrow \Rset^\dim$ is defined for every $\mathbf{x} \in \Rset^\dim$ by $\Psi(\mathbf{x}) = \mathbf{x} \, w ( \| x \|^2 )$. The function $\Psi$ is called the score function, and it is non negative since $\rho$ increases. Rewriting (\ref{Eq:psi = 0}) for each given $k \in \llbracket 1, K\rrbracket$ gives that the solution $\estCtr_k$ of this equation must verify:
\begin{equation}
\label{Eq:estCtr}
\estCtr_k =  \dfrac{\sum_{n = 1}^{N_k} w ( \| \obs_{k,n}-\estCtr_k \|^2 ) \obs_{k,n}}{\sum_{n = 1}^{N_k} w ( \| \obs_{k,n}-\estCtr_k \|^2 )},
\end{equation}
where $w$ is henceforth called the weight function. In other words, the estimate $\estCtr_k$ is a fixed point of the function
\begin{equation}
\label{Eq:gk}
g_k(\xbm) = \dfrac{\sum_{n = 1}^{N_k} w ( \| \obs_{k,n}-\xbm \|^2 )  \obs_{k,n}}{\sum_{n = 1}^{N_k} w ( \| \obs_{k,n}-\xbm \|^2 ) }, \xbm \in \Rset^\dim
\end{equation}
The computation of the fixed points of this function can be carried out by iterative algorithms described in~\cite{zoubir12robust}. 

Unfortunately, in our context, we cannot perform the M-estimation of the centroids by seeking the fixed points of each $g_k$ defined by \eqref{Eq:gk} since the number $K$ of clusters is unknown. 
However, M-estimation is supposed to be robust to outliers and, when estimating a centroid $\Ctr_k$, the measurement vectors from other clusters may be seen as outliers. 
Consequently, we can expect that if $\estCtr_k$ is a fixed point of $g_k$, then it should also be a fixed point of the function:
\begin{equation}
\label{Eq:h}
\TheFunc(\xbm) = \dfrac{\sum_{n = 1}^{N} w ( \| \obs_n-\xbm \|^2 )  \obs_n}{\sum_{n = 1}^{N} w ( \| \obs_n-\xbm \|^2 ) }, \xbm \in \Rset^\dim,
\end{equation}
where the sums involve now all the measurement vectors $\obs_n$. 
The rationale is that the contributions in $\TheFunc$ of the measurement vectors sufficiently remote from a fixed point $\estCtr_k$ should significantly be lowered by the weight function $w$, given that this weight function is robust to outliers as discussed later in the paper. 
Note that the function $h_N$ was also considered in\cite{wu02PR} for the estimation of the centroids, even though the number of clusters was assumed to be known.

At the end, the key-point for robustness to outliers of an M-estimator is the choice of the weight function $w$. ``{\em Usually, for robust M-estimators, the weights are chosen close to one for the bulk of the data, while outliers are increasingly downweighted.}'' \cite{zoubir12robust}. This standard rationale leads the authors in \cite{wu02PR} to choose the Gaussian kernel $w(\xbm) = \exp \left ( - \beta \| \xbm \|^2 \right )$, in which the parameter $\beta$ is homogeneous to a noise variance. When $d=2$, it is experimentally shown in\cite{wu02PR} that the parameter $\beta$ should be set to $1/ \std^2$. However, when $\beta > 2$, no such result exist and the value of $\beta$ must be chosen empirically (the best value of $\beta$ is usually different from $1/\std^2$).\color{black} 

\subsection{Wald p-value kernel for M-estimation}
\label{subsec: pvalue}

Here, in order to avoid the empirical choice of the parameter $\beta$, we alternatively derive our weight function from a Wald hypothesis test on the measurement vectors.

Consider the problem of testing whether two random vectors $\Obs_i$ and $\Obs_j$ ($i \neq j$) belong to the same cluster. 
This problem can be formulated as testing whether $\Ctr_i = \Ctr_j$ or not, given $\Obs_i \thicksim \Ncal(\Ctr_i,\Id)$ and $\Obs_j \thicksim \Ncal(\Ctr_j,\Id)$.
The independence  of $\Obs_i$ and $\Obs_j$ implies that $\Obs_i - \Obs_j \thicksim \Ncal(\Ctr_i - \Ctr_j, 2 \Id)$. 
Therefore, testing $\Ctr_i = \Ctr_j$ amounts to testing whether the mean of the Gaussian random vector $\Obs_i - \Obs_j$ is $0$. This is Problem (\ref{Eq:ddtpb}) with $\sigma_0 = \sqrt{2}$.

The p-value~\eqref{RDT-pvalue} of the Wald test $\Topt{}$ basically measures the plausibility that $\Obs_i$ and $\Obs_j$ belong to the same cluster. For testing the mean of $\Obs_i - \Obs_j \thicksim \Ncal(\Ctr_i - \Ctr_j, 2 \Id)$, the family of Wald tests is $\big \{ \Topt{\sqrt{2}}: \level \in (0,1) \big \}$. The p-value associated with this family of tests follows from (\ref{RDT-pvalue}) and is equal to $\thepval{\zbm} = \Marcum(0, \| \zbm \| / \sqrt{2})$ for any $\zbm \in \Rset^\dim$.
When the p-value $\thepval{\Obs_i - \Obs_j} = \Marcum(0, \| \Obs_i - \Obs_j  \| / \sqrt{2})$ is low, the null hypothesis should be rejected, that is, the two observations should be regarded as emanating from two different clusters. 
In contrast, when this p-value is large, the null hypothesis should be accepted, that is, the two observations should be considered as elements of the same clusters. 
This p-value thus basically satisfies the fundamental requirement for a weight function in robust M-estimation. Thence the idea to put $w( \| \Obs_n -  \xbm \|^2) = \thepval{\Obs_n -  \xbm} = \Marcum(0,\| \Obs_n -  \xbm \| / \sqrt{2})$ in the expression of $\TheFunc$ given by Eq. (\ref{Eq:h}). Accordingly, the weight function for the M-estimation of the centroids is thus defined as:
\begin{equation}
\label{Eq: our weight}
w(u) = \Marcum( 0,\sqrt{u /2} )
\end{equation}
for any $u \in \Rset$, so that $w(\| \zbm \|^2) = \Marcum(0, \| \zbm \|/\sqrt{2})$ for all $\zbm \in \Rset^\dim$. Because of its derivation, this weight function is hereafter called Wald p-value kernel. It follows from this choice that $\rho(u) = \int_0^u \Marcum(\tol,\sqrt{t /2}) \dt + C$ where $C$ is a constant and, as a consequence, the score function is given for any $\xbm \in \Rset^\dim$ by:
\begin{equation}
\label{Eq: Our Psi}
\Psi(\xbm) = \xbm \, \Marcum( 0, \| \xbm \| /\sqrt{2}) .
\end{equation}

From an M-estimation point of view, the relevance of the Wald p-value kernel $w$ can be further emphasized by calculating and analyzing the influence function of the M-estimator (\ref{Eq:estCtr}). 
This analysis is carried out in Appendix \ref{App: robustness} and shows that the weight function given by (\ref{Eq: our weight}) is robust to outliers.
However, this analysis does not seem to be sufficient for the clustering problem in which outliers may be numerous since they are in fact measurement vectors belonging to other clusters. 
At this stage, the question (which is not answered in~\cite{wu02PR}, even for the Gaussian kernel) is thus whether the centroids can still be expected to be fixed-points of the function $\TheFunc$ defined in~\eqref{Eq:h}. The next section brings an asymptotic answer to this question.

%

\subsection{Fixed points analysis}
\label{subsec: fixed point analysis}

The objective of this section is to support our choice of the Wald p-value kernel $w$ and to show that the centroids can be estimated by seeking the fixed-points of $\TheFunc$.
For this, the following proposition proves that the centroids $\Ctr_1,\cdots, \Ctr_K$ are the fixed-points of $\TheFunc$ when the number of measurement vectors asymptotically increases and when the centroids are asymptotically far away from each other. 
\begin{Proposition}
\label{Prop: fixed points}
Let $\Ctr_1, \ldots, \Ctr_K$  be $K$ pairwise different elements of $\Rset^\dim$. For each $k \in \llbracket 1, K \rrbracket$, suppose that $\Obs_{k,1}, \ldots, \Obs_{k,N_k} \stackrel{\text{iid}}{\thicksim} \Ncal(\Ctr_k,\Id)$ are $N_k$ independent random vectors and set $N = \sum_{k=1}^K N_k$. If there exist $\alpha_1, \ldots, \alpha_K \in (0,1)$ such that $\lim\limits_{N \to \infty} N_k/N = \alpha_k$, then, for any $i \in \llbracket 1, K \rrbracket$ and any $\Ctr$ in a neighborhood of $\Ctr_i$,
\begin{equation}\nonumber
\label{Eq: Asymptotic behavior of gN(theta)-theta-2}
\displaystyle \lim_{\forall k \ne i, \| \Ctr_k - \Ctr_i \| \rightarrow \infty}
\left ( \, 
\displaystyle \lim_{N \to \infty}
\big ( \, \TheFunc(\Ctr) - \Ctr \, \big ) 
\, \right ) = 0 
\quad \text{iff} \quad \Ctr = \Ctr_i
\end{equation}
\end{Proposition}
\begin{IEEEproof}
For any $k \in \llbracket 1, K \rrbracket$ and any $n \in \llbracket 1, N_k \rrbracket$, set $\Noise_{k,n} = \Obs_{k,n} - \Ctr_k$, so that:
$\Noise_{k,1}, \ldots, \Noise_{k,N_k} \stackrel{\text{iid}}{\thicksim} \Ncal(0,\Id).$
For each $k \in \llbracket 1, K \rrbracket$, we also set $\alpha_{k,N} = N_k/N$. 
We therefore have $\sum_{k = 1}^K \alpha_k = 1$. The random function (\ref{Eq:h}) can then be rewritten as:
\begin{equation}
\label{Eq:gk rewritten}
\TheFunc(\thetabm) =  \dfrac{U_N(\thetabm)}{V_N(\thetabm)} \quad (\thetabm \in \Rset^\dim)
\end{equation}
with
\begin{onecol}
\begin{equation}
\label{Eq: U et V}
\left \{
\begin{array}{lll}
U_N(\thetabm) = \displaystyle \sum_{k=1}^K \sum_{n = 1}^{N_k} w ( \| \Obs_{k,n}-\thetabm \|^2 )  \Obs_{k,n} \myvspace 
\\
V_N(\thetabm) = \displaystyle \sum_{k=1}^K \sum_{n = 1}^{N_k} w ( \| \Obs_{k,n}-\thetabm \|^2 ).
\end{array}
\right.
\end{equation}
\end{onecol}
\begin{twocol}
\begin{equation}
\label{Eq: U et V}
\left \{
\begin{array}{lll}
U_N(\thetabm) = \displaystyle \sum_{k=1}^K \sum_{n = 1}^{N_k} w ( \| \Obs_{k,n}-\thetabm \|^2 )  \Obs_{k,n} 
\\
V_N(\thetabm) = \displaystyle \sum_{k=1}^K \sum_{n = 1}^{N_k} w ( \| \Obs_{k,n}-\thetabm \|^2 ).
\end{array}
\right.
\end{equation}
\end{twocol}
We then have:
$$
\TheFunc(\thetabm) - \thetabm = \dfrac{W_N(\thetabm)}{V_N(\thetabm)}
$$
with $W_N(\thetabm) = U_N(\thetabm) - V_N(\thetabm) \thetabm$. By setting $\dev_k = \Ctr_k - \thetabm$ for $k \in \llbracket 1,N_k \rrbracket$, we can write:
\vspace{-0.25cm}
\begin{equation}
\label{Eq: WN/N}
\dfrac{1}{N} W_N(\thetabm) = \displaystyle \sum_{k=1}^K \alpha_{k,N} \frac{1}{N_k}\displaystyle \sum_{n=1}^{N_k} w ( \| \dev_k + \Noise_{k,n} \|^2 ) \left ( \dev_k + \Noise_{k,n} \right )
\end{equation}
In the same way:
\vspace{-0.25cm}
\begin{equation}
\label{Eq: VN/N}
\dfrac{1}{N} V_N(\thetabm) = \displaystyle \sum_{k=1}^K \alpha_{k,N} \frac{1}{N_k}\displaystyle \sum_{n=1}^{N_k} w ( \| \dev_k + \Noise_{k,n} \|^2 )
\end{equation}
By the strong law of large numbers, it follows from \eqref{Eq: WN/N} and \eqref{Eq: VN/N} that:
\begin{onecol}
\begin{equation}\nonumber
\label{Eq: Asymptotic behavior of gN(theta)-theta-2}
\displaystyle \lim_{N \to \infty} \left (\TheFunc(\Ctr) - \Ctr \right ) = 
\dfrac{\alpha_{i} \, \Expect{w ( \| \VZ{\dev_i} \|^2 ) \VZ{\dev_i} } + \displaystyle \sum_{k=1, k \ne i}^K \alpha_{k} \, \Expect{w ( \| \VZ{\dev_k} \|^2 ) \VZ{\dev_k} }}{\alpha_{i} \, \Expect{w ( \| \VZ{\dev_i} \|^2 ) } + \displaystyle \sum_{k=1, k \ne i}^K \alpha_{k} \, \Expect{w ( \| \VZ{\dev_k} \|^2 )}} \quad \text{(a-s)}
\end{equation}
with $\VZ{\dev_k} = \dev_k + \Noise$ for each $k \in \llbracket 1, K \rrbracket$ and $\Noise \thicksim \Ncal(0,\Id)$. If $\Ctr$ is in a neighborhood of $\Ctr_i$, there exists some positive real number $\varepsilon > 0$ such that $\| \dev_i \| \leqslant \varepsilon$. According to Lemma \ref{Lemma: asymptotic behaviors} stated and proved in Appendix \ref{App: asymptotic behaviors}, 
$$
\displaystyle \lim_{\forall k \ne i, \| \Ctr_k - \Ctr_i \| \rightarrow \infty}
\left ( \, 
\displaystyle \lim_{N \to \infty}
\big ( \, \TheFunc(\Ctr) - \Ctr \, \big ) 
\, \right ) = 
\dfrac{\Expect{w ( \| \VZ{\dev_i} \|^2 ) \VZ{\dev_i} }}{\Expect{w ( \| \VZ{\dev_i} \|^2 ) }} \quad \text{(a-s)}
\vspace{0.25cm}
$$
\end{onecol}
\begin{twocol}
\vspace{-0.25cm}
\begin{equation}\nonumber
\label{Eq: Asymptotic behavior of gN(theta)-theta-2}
\displaystyle \lim_{N \to \infty} \left (\TheFunc(\Ctr) - \Ctr \right ) = 
\dfrac{\alpha_{i} \, \Expect{w ( \| \VZ{\dev_i} \|^2 ) \VZ{\dev_i} } + \term_1}{\alpha_{i} \, \Expect{w ( \| \VZ{\dev_i} \|^2 ) } + \term_2} \, \, \text{(a-s)}
\end{equation}
with
\begin{align}
\nonumber
\term_1 & = \sum_{k=1, k \ne i}^K \alpha_{k} \, \Expect{w ( \| \VZ{\dev_k} \|^2 ) \VZ{\dev_k} } \\
\nonumber
\term_2 & = \sum_{k=1, k \ne i}^K \alpha_{k} \, \Expect{w ( \| \VZ{\dev_k} \|^2 )},
\end{align}
$\VZ{\dev_k} = \dev_k + \Noise$ for each $k \in \llbracket 1, K \rrbracket$ and $\Noise \thicksim \Ncal(0,\Id)$. If $\Ctr$ is in a neighborhood of $\Ctr_i$, there exists some positive real number $\varepsilon > 0$ such that $\| \dev_i \| \leqslant \varepsilon$. According to Lemma \ref{Lemma: asymptotic behaviors} stated and proved in Appendix \ref{App: asymptotic behaviors}, 
\begin{align}
\nonumber
\displaystyle \lim_{\forall k \ne i, \| \Ctr_k - \Ctr_i \| \rightarrow \infty}
&
\left ( \, 
\displaystyle \lim_{N \to \infty}
\big ( \, \TheFunc(\Ctr) - \Ctr \, \big ) 
\, \right ) \\
\nonumber
& = \dfrac{\Expect{w ( \| \VZ{\dev_i} \|^2 ) \VZ{\dev_i} }}{\Expect{w ( \| \VZ{\dev_i} \|^2 ) }} \quad \text{(a-s)}
\end{align}
\end{twocol}

Since $\Expect{w ( \| \VZ{\dev_i} \|^2 ) } > 0$, the left hand side (lhs) to the equality above is $0$ if and only if $\Expect{w ( \| \VZ{\dev_i} \|^2 ) \VZ{\dev_i} } = 0.$ The conclusion follows from Lemma \ref{Lemma: last useful lemma} in Appendix~\ref{App:D}. 
\end{IEEEproof}

The results of Proposition~\ref{Prop: fixed points} state that the centroids are the unique fixed points of the function $\TheFunc$, when the sample size and the distances between centroids tend to infinity. 
In \algo, an iterative procedure~\cite{zoubir12robust} is used to seek the fixed points of the function $\TheFunc$.
The fixed points are determined one after the other, and in order to find one fixed point, the iterative procedure is initialized with a measurement vector that has not been marked yet. 
The marking operation is applied after each centroid estimation and consists of applying a Wald test aimed at finding the measurement vectors that have the same mean as the newly estimated centroid, \pastorv{in the sense of Section~\ref{Subsec: Wald}}.
In order to define the Wald test that will be used to mark the measurement vectors, we need the statistical model of the estimated centroids. This statistical model is derived in the next section.


\vspace{-0.25cm}
\subsection{Fixed point statistical model and fusion}
\label{subsec: Statistical model and fusion}
In order to apply the Wald test for the marking operation, we need a model of the statistical behavior of the fixed points of $\TheFunc$. 
In practice, the test will be applied when the sample size and the distances between centroids may not be large enough to consider the asymptotic conditions of Proposition \ref{Prop: fixed points}. 
That is why we determine the statistical model of the fixed points from a rough estimate of their convergence rate to the centroids.


\label{subsubsec: Fixed point statistical model}
A fixed point of $\TheFunc$ provides us with an estimated centroid $\estCtr_k = \TheFunc(\estCtr_k)$ for some unknown centroid $\Ctr_k$. In order to model the estimation error, we can start by writing $\TheFunc(\estCtr_k) = \TheFunc(\Ctr_k) + W_{k,1}$. Of course, $W_{k,1}$ will be all the more small than $\estCtr_k$ approximates accurately $\Ctr_k$. We can then write that $\TheFunc(\Ctr_k) = g_k(\Ctr_k) + W_{k,2}$, where $W_{k,2}$ is likely to be small since the weighted averaging performed by $h_N$ downweights data from clusters other than $k$. In absence of noise, we would directly have $g_k(\Ctr_k) = \Ctr_k$. The presence of noise induces that $g_k(\Ctr_k) = \Ctr_k + W_{k,3}$. At the end, we have $\estCtr_k  = \Ctr_k + W_{k,3} + W_{k,2} + W_{k,1}$. According to Proposition \ref{Prop: fixed points}, $W_{k,1}$ and $W_{k,2}$ can be expected to remain small if centroids are far from each other. 
Unfortunately, we cannot say more about these two terms and we do not know yet how to model them. In contrast, it turns out that the term $W_{k,3}$ can be modeled as follows. 

For any $k \in \llbracket 1, K\rrbracket$ and any $n \in \llbracket 1, N_k\rrbracket$, set $\Noise_{k,n} = \Obs_{k,n} - \Ctr_k$. With the same notation as above, it follows from (\ref{Eq:gk}) that:
\vspace{-0.3cm}
\[W_{k,3} = g_k(\Ctr_k) -  \Ctr_k = \dfrac{\displaystyle \sum_{n = 1}^{N_1} w ( \| \Noise_{k,n} \|^2 ) \Noise_{k,n}}{\displaystyle \sum_{n = 1}^{N_k} w ( \| \Noise_{k,n} \|^2 )}\]
By the central limit theorem and from Appendix \ref{App: Correlation matrix}, we get $$\dfrac{1}{\sqrt{N_k}} \displaystyle \sum_{n = 1}^{N_k} w ( \| \Noise_{k,n} \|^2 ) \Noise_{k,n} \mathop{\Rightarrow}_{N_k \to \infty} \Ncal \left ( 0,\Expect{w \left ( \| \Noise \|^2 \right )^2} \Id \right )$$
where $\Noise \thicksim \Ncal(0,\Id)$. On the other hand, the weak law of large numbers yields:
$$\dfrac{1}{N_k} \displaystyle \sum_{n = 1}^{N_k} w ( \| \Noise_{k,n} \|^2 ) \stackrel{\Pbb}{\rightarrow} \Expect{w \left ( \| \Noise \|^2 \right )}$$ 
Slutsky's theorem \cite[Sec. 1.5.4, p. 19]{Serfling1980}
\color{black} then implies:
$$
\dfrac{\sqrt{N_k} \displaystyle \sum_{n = 1}^{N_k} w ( \| \Noise_{k,n} \|^2 ) \Noise_{k,n}}{\displaystyle \sum_{n = 1}^{N_k} w ( \| \Noise_{k,n} \|^2 )} \mathop{\Rightarrow}_{N_k \to \infty} \Ncal \left ( 0,\NewStd^2 \Id \right )
$$
with 
\begin{equation}\label{eq:rho}\NewStd^2 = \dfrac{\Expect{w \left ( \| \Noise \|^2 \right )^2}}{\Expect{w \left ( \| \Noise \|^2 \right )}^2}\end{equation}
Therefore, $W_{k,3}$ is asymptotically Gaussian so that: $g_k(\Ctr_k) = \Ctr_k + W_{k,3} \thicksim \ANcal \left ( \Ctr_k, ({\NewStd^2}/{N_k}) \, \Id\right )$
Since we do not know how to characterize $W_{k,1}$ and $W_{k,2}$, whose contributions can however be expected to be small, we opt for a model that does not take the influence of these terms into account. Therefore, we model the statistical behavior of $\estCtr_k$ by setting:
\begin{equation}
\label{Eq: model of estimated centroids}
\estCtr_k \thicksim \Ncal(\Ctr_k, (\NewStd^2/N_k) \, \Id ).
\end{equation} 
This model discards the influence of $W_{k,1}$ and $W_{k,2}$, but the experimental results reported in Section \ref{sec:experiments} support the approach. 

\subsection{Marking operation}
\label{subsec: test for clustering}

We now define the hypothesis test that permits to mark the measurement vectors given the estimated centroids.
Let $\estCtr_k$ be an estimate of the unknown centroid $\Ctr_k$. We assume that $\estCtr_k$ follows the model of Eq. (\ref{Eq: model of estimated centroids}). Consider a measurement vector $\Obs_n$ and let $\Ctr_\ell$ be the centroid of the unknown cluster to which $\Obs_n$ belongs, so that $\Obs_n \thicksim \Ncal(\Ctr_\ell,\Id)$. The clustering problem can then be posed as the problem of testing whether $\Ctr_k = \Ctr_\ell$ or not. According to our model for $\Obs$ and $\estCtr_k$, $$\Obs_n - \estCtr_\ell \thicksim \Ncal(\Ctr_\ell-\Ctr_k, (1+(\NewStd^2/N_k)) \Id)$$ Therefore, the problem of testing whether $\Ctr_k = \Ctr_\ell$ amounts to testing the mean of the Gaussian random vector $\Obs_n - \estCtr_\ell$. According to Subsection \ref{Subsec: Wald}, the optimal spherically invariant and UBCP test for this problem is the Wald test $\Topt{\sqrt{(1+(\NewStd^2/N_k))}}$.
However, here, the value of $N_k$ is unknown since the objective of this step is to determine the measurement vectors that belong to the cluster. 
That is why we assume that $N_k$ is high enough and simply apply the Wald test $\Topt{}$ in order to perform the marking.

\subsection{Estimated centroid fusion}

Despite the marking operation, it may occur that some centroids are estimated several times with different initializations. 
In order to merge these centroids that are very close to each other, we now introduce the fusion step that is applied when all the centroids have been estimated.

Consider two estimates $\estCtr_k$ and $\estCtr_\ell$ of two unknown centroids $\Ctr_k$ and $\Ctr_\ell$, respectively. The fusion between $\estCtr_k$ and $\estCtr_\ell$ can then be posed as an hypothesis testing problem where the null hypothesis is $\Hcal_0: \Ctr_k = \Ctr_\ell$ and the alternative is $\Hcal_1: \Ctr_k \neq \Ctr_\ell$. In order to derive a solution to this binary hypothesis testing problem, we resort to the probabilistic model~\eqref{Eq: model of estimated centroids} for the estimated centroids. In this respect, we assume that $\estCtr_k \thicksim \Ncal(\Ctr_k, (\NewStd^2/N_k) \, \Id )$ and that $\estCtr_\ell \thicksim \Ncal(\Ctr_\ell, (\NewStd^2/N_\ell) \, \Id )$. In this model, 
\small
$$\estCtr_k - \estCtr_\ell \thicksim \Ncal \left ( \Ctr_k - \Ctr_\ell , \sigma_{k,\ell}^2 \Id \right )$$
\normalsize
with 
\small
\begin{equation}\label{eq:sigma_fusion}
\sigma_{k,\ell} = \NewStd \, \sqrt{\left ( \dfrac{1}{N_k} + \dfrac{1}{N_\ell} \right )}
\end{equation}
\normalsize
The testing of $\Hcal_0$ against $\Hcal_1$ then amounts to testing the mean of the Gaussian random vector $\estCtr_k - \estCtr_\ell$. According to Section \ref{Subsec: Wald}, this testing problem can optimally be solved by the Wald test $\Topt{\sigma_{k,\ell}}$. Note that in~\eqref{eq:sigma_fusion}, the values $N_k$ and $N_\ell$ are assumed to be known. In practice, since after estimation of $\Ctr_k$ (resp. $\Ctr_\ell$), the measurement vectors close enough to $\Ctr_k$ (resp. $\Ctr_\ell$) are marked, we consider that the number of these marked vectors approximates sufficiently well $N_k$ (resp. $N_\ell$). 

Finally, it is worth emphasizing that the model considered above for centroid fusion is more restrictive than a model that would involve the distribution of the whole sum $W_{k,1} + W_{k,2} + W_{k,3}$, if this distribution were known. By writing that, we mean that the Wald test $\Topt{\sigma_{k,\ell}}$, constructed for $W_{k,3}$ only, is likely to yield more false alarms than a test exploiting the yet unknown whole distribution. In other words, $\Topt{\sigma_{k,\ell}}$ may not decide to merge estimated centroids that should be. However, the experimental results of Section \ref{sec:experiments} support the idea that the Wald test $\Topt{\sigma_{k,\ell}}$ is actually sufficient and efficient for the fusion.

\section{Clustering Algorithms}\label{sec:algo}
The objective of this section is to gather all the estimation functions and hypothesis tests defined in the previous section in order to build two clustering algorithms.
We first describe the centralized algorithm (\algo), and then derive the decentralized version (\dalgo) of the algorithm.

\subsection{Centralized clustering algorithm (\algo)}
The objective of the algorithm is to divide the set of measurement vectors $\Scal = \{\obs_1, \ldots ,\obs_N\}$ into clusters. The centralized algorithm performs the following steps. 

\subsubsection{Initialization}
Let us denote by $\Mk$ the set of vectors $\obs_k$ that are considered as marked, where marked vectors cannot be used anymore to initialize the estimation of a new centroid. The set  $\Mk$ is initialized as $\Mk = \{ \varnothing\}$. Also, let $\Ectr$ be the set of centroids estimated by the algorithm, where  $\Ectr$ is initialized as $\Ectr =  \{ \varnothing\} $.
Fix a parameter $\epsilon$ that corresponds to a stopping criterion in the estimation of the centroids.

\subsubsection{Estimation of the centroids}
The centroids are estimated one after the other, until $\Mk = \Scal$. 
When the algorithm has already estimated $k$ centroids denoted $\widehat{\Ctr}_1, \cdots, \widehat{\Ctr}_k $, we have that $\Ectr = \{ \widehat{\Ctr}_1, \cdots, \widehat{\Ctr}_k   \}$. 
In order to estimate the $k+1$-th centroid, the algorithm picks a measurement vector $\obs_{\star}$  at random in the set $\in \Scal \setminus \Mk$  and initializes the estimation process with $\widehat{\Ctr}_{k+1}^{(0)} = \obs_{\star} $.
It then produces an estimate of the centroid by computing recursively $\widehat{\Ctr}_{k+1}^{(\ell+1)} = \TheFunc(\widehat{\Ctr}_{k+1}^{(\ell)}) $, where the function $\TheFunc $ was defined in~\eqref{Eq:h} and the recursion stops when $ \|  \widehat{\Ctr}_{k+1}^{(\ell+1)} - \widehat{\Ctr}_{k+1}^{(\ell)}\|_2 \leq \epsilon $.
Once the stopping condition is reached after, say, $L$ iterations, the newly estimated centroid is given by $\widehat{\Ctr}_{k+1} = \widehat{\Ctr}_{k+1}^{(L)}$, and the set of estimated centroids is updated as $\Ectr = \Ectr \cup \{ \widehat{\Ctr}_{k+1}\}$.

Once the centroid $\widehat{\Ctr}_{k+1}$ is estimated, the algorithm marks and stores in set $\Mk_{k+1}$ all the vectors that the Wald test $\Topt{}$ defined in~\eqref{Eq:Thresholding test from above} accepts as elements of cluster $k+1$. Therefore, $\Mk_{k+1} = \{ \obs_i \in \Scal : \Topt{}(\obs_i - \widehat{\Ctr}_{k+1}) = 0 \}$.
Note that a vector $\obs_i$ may belong to several sets $\Mk_{k}$, which is not an issue since a set $\Mk_{k}$ of marked vectors is not the final set of vectors assigned to the cluster (see the classification step of the algorithm).
The algorithm finally updates the set of marked vectors as $\Mk \leftarrow \Mk \cup \{ \obs_{\star} \} \cup \Mk_{k+1}$. Note that the measurement vector $ \obs_{\star}$ that serves for initialization is also marked in order to avoid initializing again with the same vectors.
If $\Mk \neq \Scal$, the algorithm estimates the next centroid $\widehat{\Ctr}_{k+2}$. Otherwise, the algorithm moves to the fusion step. 

\subsubsection{Fusion}
Once $\Mk = \Scal $ and, say, $K'$ centroids have been estimated, the algorithm applies the hypothesis test $\Topt{\sigma_{k,\ell}}$ defined in Section~\ref{subsec: Statistical model and fusion} to all pairs $(\widehat{\Ctr}_{k_1}, \widehat{\Ctr}_{k_2}) \in \Ectr \times \Ectr$ such that $k_1 \neq k_2$. 
We assume without loss of generality that the indices $k_1$ and $k_2$ are chosen such that $k_1 < k_2$.
When $\Topt{\sigma_{k,\ell}}(\widehat{\Ctr}_{k_1} - \widehat{\Ctr}_{k_2}) = 0$, the algorithm sets $ \widehat{\Ctr}_{k_1} = \frac{\widehat{\Ctr}_{k_1} + \widehat{\Ctr}_{k_2}}{2}$ and removes $ \widehat{\Ctr}_{k_2}$ from $\Ectr$. 

The fusion step ends when $\Topt{\sigma_{k,\ell}}(\widehat{\Ctr}_{k_1} - \widehat{\Ctr}_{k_2}) = 1$ for all the $(\widehat{\Ctr}_{k_1}, \widehat{\Ctr}_{k_2}) \in \Ectr \times \Ectr$ such that $k_1 \neq k_2$. At this stage, the algorithm sets the number of centroids $K$ as the cardinal of $\Ectr$ and re-indexes the elements of $\Ectr$ in order to get $\Ectr = \{\widehat{\Ctr}_1, \cdots \widehat{\Ctr}_K\} $. It then moves to the final classification step.

\subsubsection{Classification}\label{subsec:classif_algo}
Denote by $\Ck$ the set of measurement vectors assigned to cluster $k$. At the classification step, $\Ck$ is initialized as $\Ck = \{ \varnothing\}$.
Each vector $\obs_i \in \Scal$ is then assigned to the cluster $\mathcal{C}_{k'}$ whose centroid $\widehat{\Ctr}_{k'} \in \Ectr$ is the closest to $\obs_i$, that is $\widehat{\Ctr}_{k'} = \arg\min_{\widehat{\Ctr} \in \Ectr} \| \obs_i - \widehat{\Ctr} \| $. 
%
Note that $\mathcal{C}_k$ can be different from $\mathcal{M}_k$, due to the fusion step, but also to the closest centroid condition used during the classification.
In particular, each vector $\obs_i$ can belong to only one single $\Ck$.

In this version of the algorithm, we classify the measurement vectors by using the minimum distance condition. This ends up to assigning each and every measurement vector to a cluster. However, it is worth noticing that the algorithm could easily be modified by classifying the measurement vectors via the Wald hypothesis test (as for the marking process). By so proceeding, a measurement vector would be assigned to a centroid only if it is sufficiently close to this one. This would permit to detect outliers as vectors that have not been assigned to any cluster.


\subsection{Decentralized clustering algorithm (\dalgo)}
In the decentralized algorithm, the operations required by the algorithm are performed by the sensors themselves over the data transmitted by the other sensors. 
Each of the $N$ sensors has access to one single measurement vector $\obs_n$, only.
We assume that the transmission link between two sensors is perfect, in the sense that no error is introduced during information transmission.
We now describe the decentralized version of the algorithm, and point out the differences with the centralized algorithm.

\subsubsection{Initialization of the algorithm}
In the distributed algorithm, each sensor $n \in \{1,\cdots, N\}$ produces its own set of centroids, denoted $\Ectr_n$ and initialized as $\Ectr_n =  \{ \varnothing\} $.
Since each sensor only has its own observation $\obs_n$, each sensor $n$ now has its own marking variable $M_n$ initialized as $M_n = 0$. 
Denote by $T$ the number of time slots available for the estimation of each centroid, and by $0\leq L\leq N$ a stopping condition for the centroid estimation. 

\subsubsection{Estimation of a centroid}
As for the centralized version, the centroids are estimated one after the other, until $M_n = 1$ for all $n \in \{1,\cdots, N\}$.
When sensor $n$ has already estimated $k$ centroids denoted $\widehat{\Ctr}_{n,1}, \cdots, \widehat{\Ctr}_{n,k} $, we have that $\Ectr_n = \{ \widehat{\Ctr}_{n,1}, \cdots, \widehat{\Ctr}_{n,k}   \}$. 
All the sensors produce their $k+1$-th estimates $\widehat{\Ctr}_{n,k+1} $ at the same time as follows.  

For the initialization of the centroid estimation, one sensor $n'$ is selected at random among the set of sensors $n$ for which $M_n= 0$.
The vector $\obs_{n'}$ observed by this sensor is broadcasted to all the other sensors.
Each sensor $n$ initializes its estimated centroid as $\widehat{\Ctr}_{n,k+1}^{(0)} = \obs_{n'}$, as well as two partial sums $P_{n} = w(\| \obs_n - \widehat{\Ctr}_{n,k+1}^{(0)} \|^2) \obs_n$, $Q_{n} = w(\| \obs_n - \widehat{\Ctr}_{n,k+1}^{(0)} \|^2)$.
It also initializes a counter of the number of partial sums $c_n = 0$ received by sensor $n$.

At each time slot $t=0,\cdots, T$, sensor $n$ receives $J$ partial sums from $J$ other sensors ($J$ can vary from time slot to time slot).  We denote the partial sums received by sensor $n$ as $P_{1\rightarrow n}^{(t)}, \cdots, P_{J\rightarrow n}^{(t)}$ and $Q_{1\rightarrow n}^{(t)}, \cdots, Q_{J\rightarrow n}^{(t)} $. 
We assume that the sensor also receives the counters $c_1, \cdots, c_J$ \pastorv{of partial sums calculated by the $J$ other sensors}. 
The sensor then updates its partial sums as 

\vspace{-0.4cm}

\begin{equation} \nonumber
 P_n^{(t)}  = P_n^{(t-1)} + \sum_{j=1}^J P_{j\rightarrow n}^{(t)}, ~~~
 Q_n^{(t)} = Q_n^{(t-1)} + \sum_{j=1}^J Q_{j\rightarrow n}^{(t)}
\end{equation}
and its counter as $c_n \leftarrow c_n + \sum_{j=1}^J c_j$.
Afterwards, if $c_n \geq L$, the sensor updates its estimate $\widehat{\Ctr}_{n,k+1}$ as $\widehat{\Ctr}_{n,k+1} = \frac{P_n^{(t)}}{Q_n^{(t)}}$ and reinitializes its partial sums as $P_n^{(t)} \leftarrow w(\| \obs_n - \widehat{\Ctr}_{n,k+1} \|^2) \obs_n$, $Q_n^{(t)} \leftarrow w(\| \obs_n - \widehat{\Ctr}_{n,k+1} \|^2)$, and its counter as $c_n \leftarrow 0$. As long as $c_n \leq L$, the sensor has not received enough partial sums from the other sensors: it does not update its estimated centroid and waits for the next time slot. 
In the above process, the centroids are estimated from $\frac{P_n^{(t)}}{Q_n^{(t)}}$, which correspond to an approximation of the function $h_N$ defined in~\eqref{Eq:h}. 

The estimation process stops when time slot $T$ is reached.  At this time, sensor $n$ updates its set of centroids as $\Ectr_n \leftarrow \Ectr_n \cup \{ \widehat{\Ctr}_{n,k+1} \} $. 
If then verifies whether its observation $\obs_n$ belongs to the newly created cluster by applying the same hypothesis test as in \algo. In this case, if $ \Topt{}(\obs_i - \widehat{\Ctr}_{k+1}) = 0$, then $M_n = 1$. 
If $M_n = 1$ for all $n \in \{1,\cdots, N\}$, the algorithm moves to the next step.

\subsubsection{Fusion}
The fusion step is almost the same as in the centralized algorithm, except that each sensor performs its own fusion over the set $\Ectr_n$. 
Sensor $n$ applies the hypothesis test $ \Topt{\sigma_{k,\ell}}(\widehat{\Ctr}_{n,k_1} - \widehat{\Ctr}_{n,k_2})$ defined in Section~\ref{subsec: Statistical model and fusion} to all the $(\widehat{\Ctr}_{n,k_1}, \widehat{\Ctr}_{n,k_2}) \in \Ectr_n \times \Ectr_n$ such that $k_1 \neq k_2$. 
It then fusions the centroids $\widehat{\Ctr}_{n,k_1}$,  $\widehat{\Ctr}_{n,k_2}$ for which $\Topt{\sigma_{k,\ell}}(\widehat{\Ctr}_{n,k_1} - \widehat{\Ctr}_{n,k_2}) = 0$ and denotes by $K_n$ the final number of centroids in $\Ectr_n$.

\subsubsection{Classification}
The classification step is exactly the same as in the centralized algorithm, except that sensor $n$ only classifies its own observation $\obs_n$.
The sensor identifies the centroid $\widehat{\Ctr}_{n,k'} \in \Ectr_n$ that is the closest to $\obs_n$, that is $\widehat{\Ctr}_{n,k'} = \arg\min_{\widehat{\Ctr} \in \Ectr_n} \| \obs_n - \widehat{\Ctr} \| $. 

At the end, the proposed decentralized algorithm induces more latency in the clustering compared to K-means, since in K-means the centroids are estimated in parallel. However, our algorithm should reduce the overall number of messages exchanged between sensors, since it does not need to be initialized by the K-means++ procedure and since it does not need to estimate the number of clusters.


\section{Experimental results}\label{sec:experiments}
This section evaluates the performance of \algo~and \dalgo~through Monte Carlo simulations. 
In all our simulations, the observation vectors $\obs_n$ that belong to cluster $k$ are generated according to the model $\Obs_n \sim \mathcal{N}(\Ctr_k, \sigma^2 \Id)$, where $\sigma^2$ is the noise variance. 
Depending on the considered setup, the centroids $\Ctr_k$ will be generated differently. In order to evaluate the performance of our algorithm, we consider two figures of merit:
\begin{itemize}
 \item[-] The \emph{classification error probability $P_e$} estimates the probability that a data point has been assigned to the wrong cluster. It is determined  from the confusion matrix which is a 2D matrix with true clusters in line and estimated clusters in columns. 
 In cell $(i,j)$ of the matrix is indicated the percentage of data from estimated cluster $j$ that actually belong to estimated cluster $i$. 
 \item[-] The \emph{estimation distortion $D$} is defined as the average distance between the data points and the estimated centroids of the clusters to which they have been assigned. 
\end{itemize}
In all our simulations, the parameter $\NewStd^2$~\eqref{eq:rho} used for the fusion is evaluated numerically from Monte Carlo simulations by averaging over $10000$ realizations of $\mathbf{X}$. Note that this parameter depends on the dimension but depends on neither $\std$ nor the considered data. It is thus computed once for all for every set of simulations. 
The probability of false alarm $\level$ is always set to $10^{-3}$ and the stopping criterion $\epsilon$ is always set to $10^{-2}$. However, these empirical parameters do not influence much the performance of the algorithm as long as they belong to a reasonable range (roughly, from $10^{-6}$ to $10^{-2}$).

\subsection{Centralized algorithm}
We start by evaluating the performance of the centralized version of the algorithm.
For comparison, we evaluate the performance of the K-means algorithm, and we consider two methods in order to alleviate the sensitivity of K-means to initialization. 
We first use the method of replicates, that involves running the algorithm $R$ times with random initializations and choosing the solution that minimizes the average distance between the measurement vectors and the estimated centroids. In all our simulations, we consider $R=10$ and $R=100$.
Second, we consider the K-means++ algorithm without any replicates. 
As discussed earlier in the paper, these two methods may not be appropriate for a distributed treatment.
However, our purpose here is to assess the performance of our centralized algorithm compared to existing algorithms, before considering a distributed context.

In our experiments, we also assume that the number $K$ of clusters is provided to the K-means algorithm. 
This choice places K-means in very favorable conditions, but permits to avoid running the algorithm several times with various values of $K$. 
Note that all the K-means algorithms we use in our experiments are obtained from the Matlab function ``kmeans''. 
We now evaluate the performance of our centralized algorithm against K-means for various setups described in the following.

\subsubsection{ Parameters $\dim = 2$, $K = 4$, $N = 400$}\label{subsec:exp_dim2}


\begin{figure*}[t]
\begin{center}
  \subfloat[~]{ \includegraphics[width=.32\linewidth]{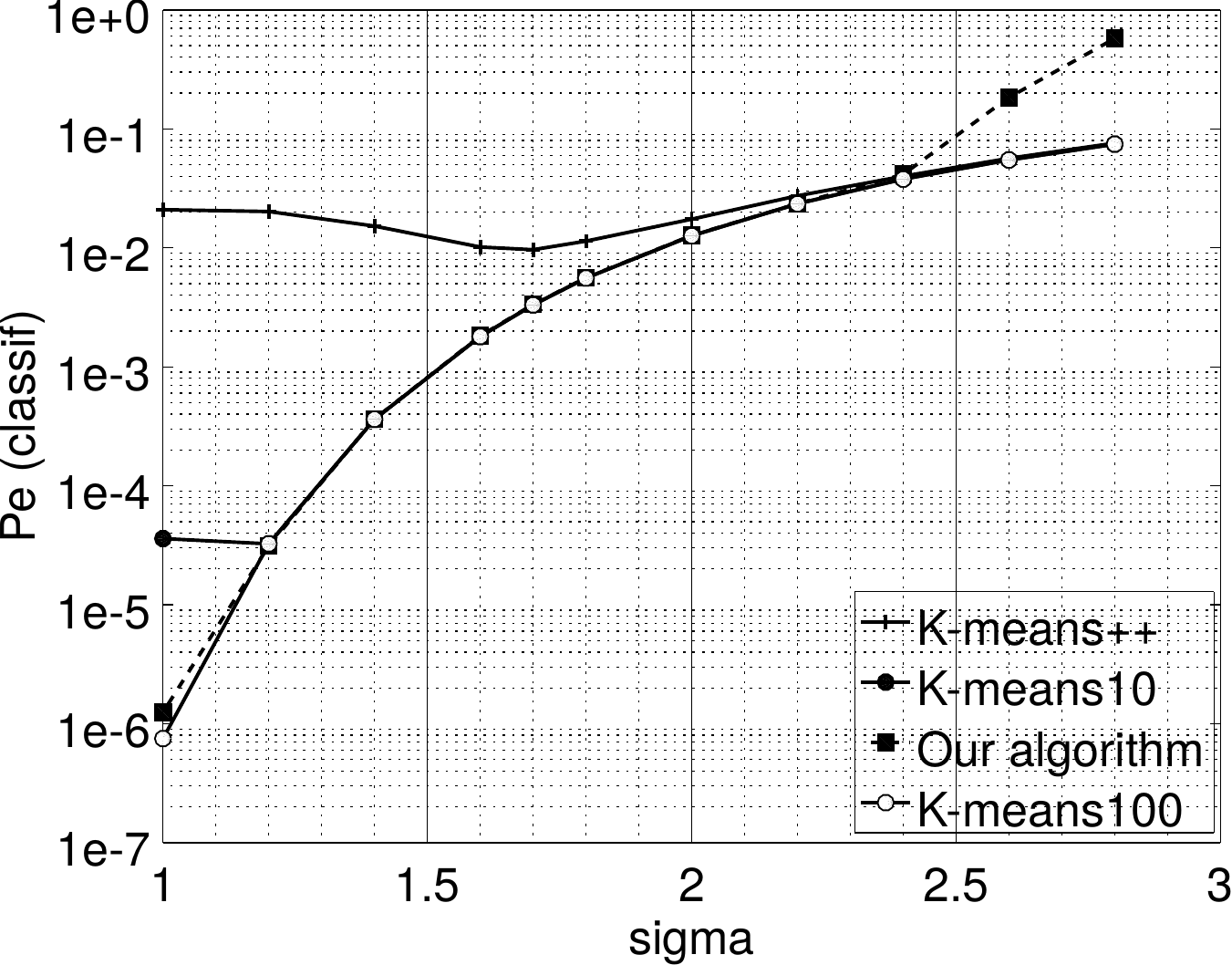}}
  \subfloat[~]{ \includegraphics[width=.32\linewidth]{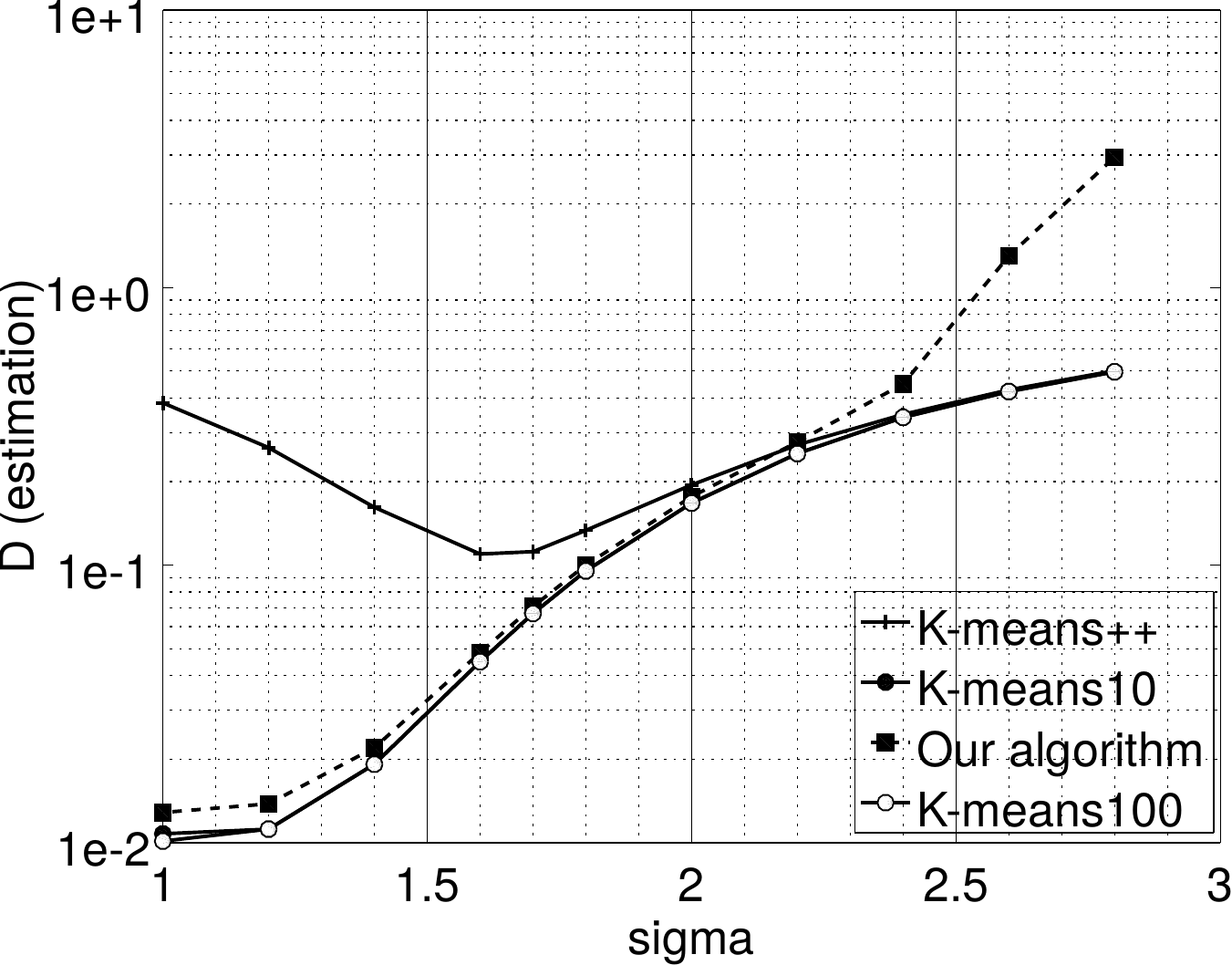}}
  \subfloat[~]{\includegraphics[width=0.32\textwidth]{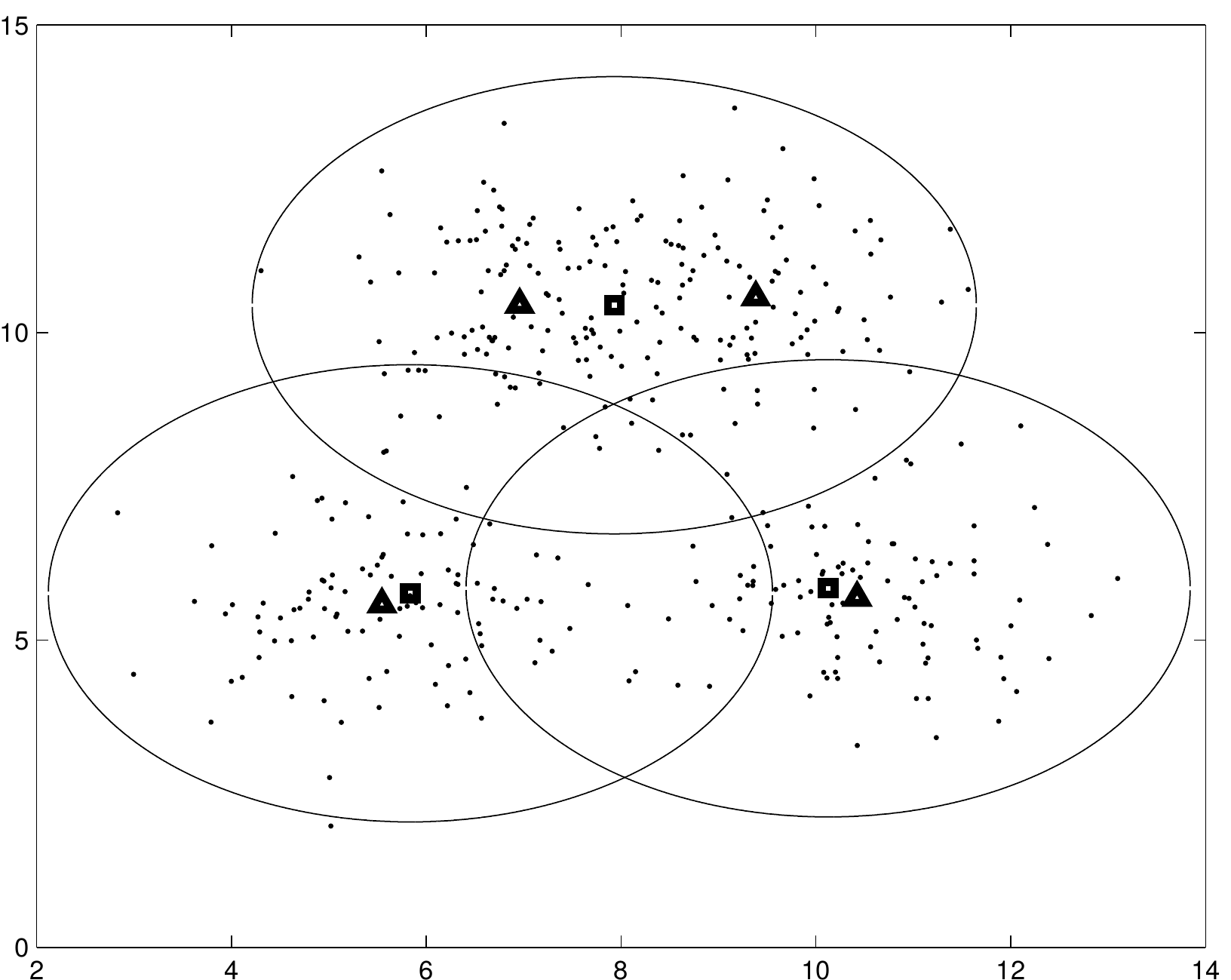}}
  \vspace{-0.2cm}
\end{center}
\caption{Performance evaluation of \algo, with $\dim = 2$, $K = 4$, $N = 400$: (a) Classification error probability (b) Estimation distortion. In both Figures, Kmeans10 and Kmeans100 correspond to the K-means algorithm with $10$ and $100$ replicates, respectively. (c) Example of clustering when $\sigma>\sigma_{\text{lim}}$. Triangles give the centroids estimated by K-means10 and squares give centroids estimated by our algorithm. The circles correspond to the decision thresholds. \vspace{-12pt}}
\label{fig:res_dim2}
\end{figure*}

In this part, the observations are generated as follows. The vector dimension is given by $d=2$, the number of clusters is $K=4$, and the number of observations is $N=400$.
The $4$ centroids are given by $\Ctr_1 = [A,2A]$, $\Ctr_2 = [2A, A]$, $\Ctr_3 = [A, A]$, $\Ctr_4 = [2A, 2A]$, where $A = 10$.
The number $N_k$ of observations in cluster $k$ does not depend on $k$ and is given by $N_k=N/K = 100$. 
We consider various values of $\sigma$, and evaluate the performance of the centralized algorithms over $Nt = 10000$ realizations for each considered value of $\sigma$. 

The results are presented in Figure~\ref{fig:res_dim2} (a) for the classification error probability and in Figure~\ref{fig:res_dim2} (b) for the estimation distortion.
First, it is worth noticing than K-means++ without any replicates does not perform well in this setup. 
In particular, for K-means++, low values of $\sigma$ unexpectedly give higher classification error probability and estimation distortion. 
This is probably due to the fact that when $\sigma$ is low, the clusters are so far from each other that a bad initialization cannot be handled by the algorithm. 
As a result, K-means++ alone may not be sufficient by itself and may need replicates as well.

Further, according to the two figures, \algo~exhibits the same  performance as K-means algorithms $10$ and $100$ replicates, when $\sigma$ ranges from $1.2$ to $2.5$.
Below $\sigma = 1.2$, K-means with $10$ replicates, henceforth denoted by K-means10, performs worse than \algo, probably for the same initialization issues as those incurred by K-means++.
This issue does not appear anymore for K-means with $100$ replicates, since the higher number of replicates increases the chances to find a good initialization.
Therefore, even when $K$ is known, K-means requires a large number of replicates in order to increase the probability of initializing correctly the algorithm.

On the other hand, for $\sigma>2.5$, the K-means algorithm with $100$ replicates, hereafter denoted by K-means100, outperforms our algorithm.
In fact, when $\sigma$ becomes too large, two clusters can be so entangled that discriminating between them is hardly feasible without prior knowledge of the number of clusters.
In this case, K-means may still be able to separate between two clusters since it already knows that there are two clusters, while our algorithm may tend to merge the two centroids due to the value of the variance.

It turns out that we can predict the value $\sigma_{\text{lim}}$ above which two clusters may hardly be distinguished by \algo. Basically, if the balls $B(\Ctr_k,\TheThreshold{\sigma})$ and $B(\Ctr_\ell,\TheThreshold{\sigma})$ --- with same radius $\TheThreshold{\sigma}$ and respective centers $\Ctr_k$ and $\Ctr_\ell$ --- actually intersect, there might be some ambiguity in classifying elements of this intersection. 
Classifying  data belonging to this intersection will be hardly feasible as soon as $\sigma$ is such that $\| \Ctr_k - \Ctr_\ell \| \leqslant \TheThreshold{\sigma}$.
We conclude from the foregoing that our algorithm should perform well for $\sigma \leqslant \sigma_{\text{lim}}$ and may severely degrade for $\sigma > \sigma_{\text{lim}}$, where 
\begin{longversion}
\begin{equation}
\label{eq: sigmalim}
\sigma_{\text{lim}} = \min_{k,\ell} \| \Ctr_k - \Ctr_\ell \| / \mu(\gamma) .
\end{equation}
\end{longversion}
\begin{shortversion}
$\sigma_{\text{lim}} = \min_{k,\ell} \| \Ctr_k - \Ctr_\ell \| / \mu(\gamma) .$
\end{shortversion}
According to this rationale, it follows from the values chosen for $A$ and $\level$ that $\sigma_{\text{lim}} = 2.7$, which is close to the value $\sigma = 2.5$ found by considering the experimental results of Figure~\ref{fig:res_dim2}.


In order to illustrate the behavior of our algorithm for $\sigma>\sigma_{\text{lim}}$, we consider the following setup.
We keep the parameters $\dim = 2$, $K = 4$, $N = 400$, but we modify the centroids as $\Ctr_1 = [13,20]$, $\Ctr_2 = [20, 10]$, $\Ctr_3 = [10, 10]$, $\Ctr_4 = [17, 20]$.
For these parameters, $\sigma_{\text{lim}}=1.1$. We set $\sigma=2>\sigma_{\text{lim}}$ and apply both \algo~and K-means10 to the set of data.
The results are represented in Figure~\ref{fig:res_dim2} (c). We see that K-means retrieves four clusters, which is in accordance with the ground truth, while \algo finds three clusters only.
However, by taking a look at the generated data, it does not seem such a bad choice in this situation to consider three clusters instead of four, which cannot actually be assessed by the classification error probability since this one performs a comparison to the true generated clusters.

\begin{longversion}
\subsubsection{Comparison with Gaussian kernel}


\begin{figure*}[t]
\begin{center}
  \subfloat[~]{ \includegraphics[width=.48\linewidth]{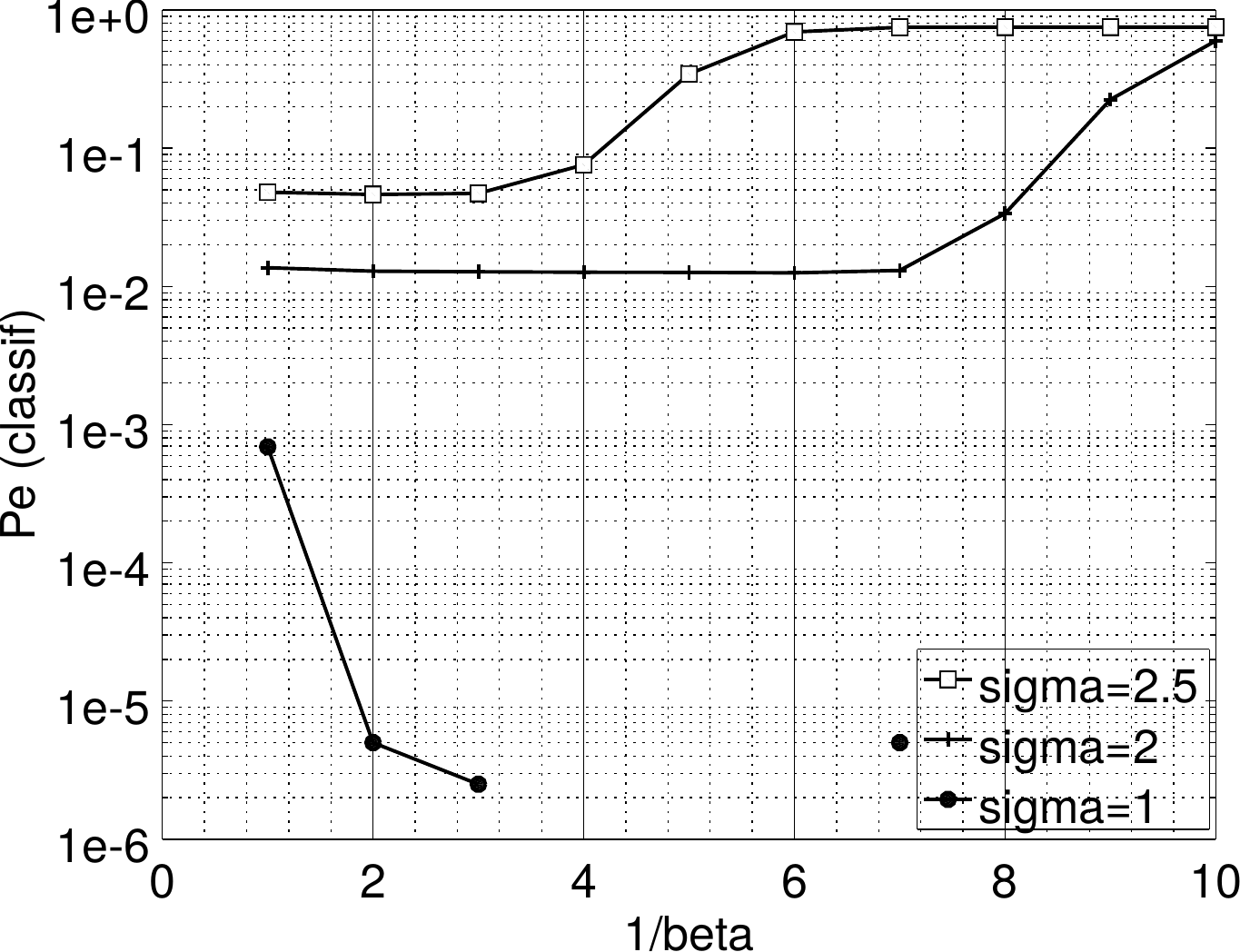}}
  \subfloat[~]{ \includegraphics[width=.48\linewidth]{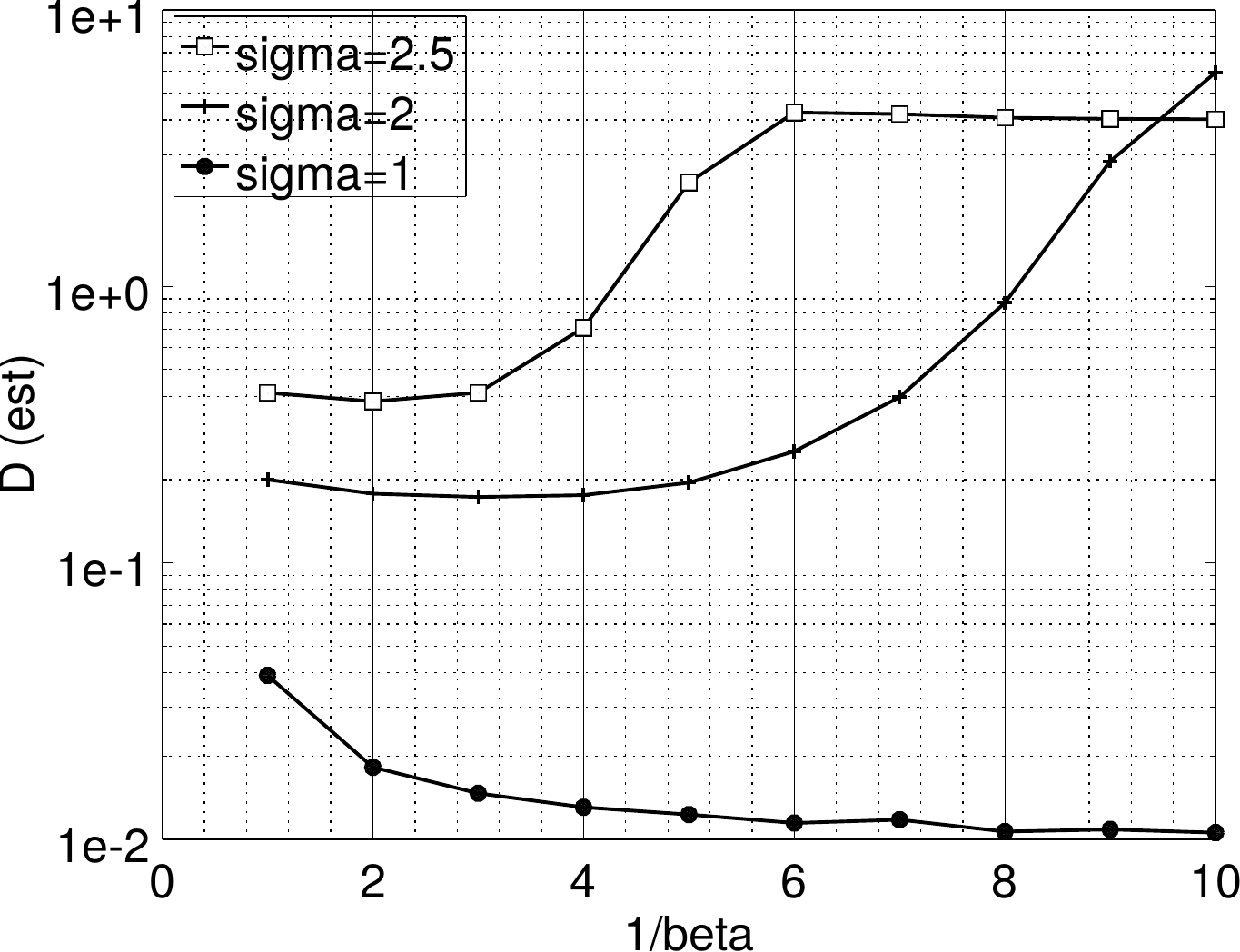}}
\end{center}
\caption{Performance of \algo~with Gaussian kernel, with respect to parameter $\beta$ and for various values of $\sigma$: (a) Classification error probability (b) Estimation distortion. \pastorv{For $1/\beta > 3$ and $\sigma = 1$, no value for the classification error probability is reported in subfigure (a), because \algo~with Gaussian kernel commited no classification error in this range.}}
\label{fig:res_exp_beta}
\end{figure*}

Here, we consider the same parameters $\dim = 2$, $K = 4$, $N = 400$ and the same data generation as in the previous experiment. 
We want to compare the performance of \algo~when, for M-estimation of the centroids, the Wald p-value kernel is replaced by the Gaussian kernel $w(\xbm) = \exp \left ( - \frac{\beta}{\sigma^2} \| \xbm \|^2 \right )$. This form for the Gaussian kernel is conveniently chosen here, instead of $w(\xbm) = \exp \left ( - \beta \| \xbm \|^2 \right ) $ as in~\cite{wu02PR}, so as to hereafter consider parameter values that are directly proportional to $1/\sigma^2$.


As explained in Section~\ref{sec:mestimator}, the main drawback of the Gaussian kernel resides in the choice of the value of $\beta$. 
In order to select this value, we first evaluated the performance of \algo~with the Gaussian kernel for various values of $\beta$ and $\sigma$.  
Figures~\ref{fig:res_exp_beta} (a) and (b) represent the obtained classification error probability and estimation distortion. 
We first see that the choice of the value of $\beta$ does not depend on whether we want to optimize the classification error criterion or the estimation distortion criterion. For example, for $\sigma=2.5$, the value $\beta=1/2$ yields good performance with respect to both criteria.
On the other hand, the optimal value of $\beta$ turns out to depend on the value of $\sigma$. For instance, for $\sigma=1$, we would select $\beta$ as small as possible, whereas we should choose a large value of $\beta$ for $\sigma=2.5$. From these observations, we can conclude that it is difficult to optimize the parameter $\beta$ once for all, for the whole range of possible values of $\sigma$. 

\begin{figure*}[t]
\begin{center}
  \subfloat[~]{ \includegraphics[width=.48\linewidth]{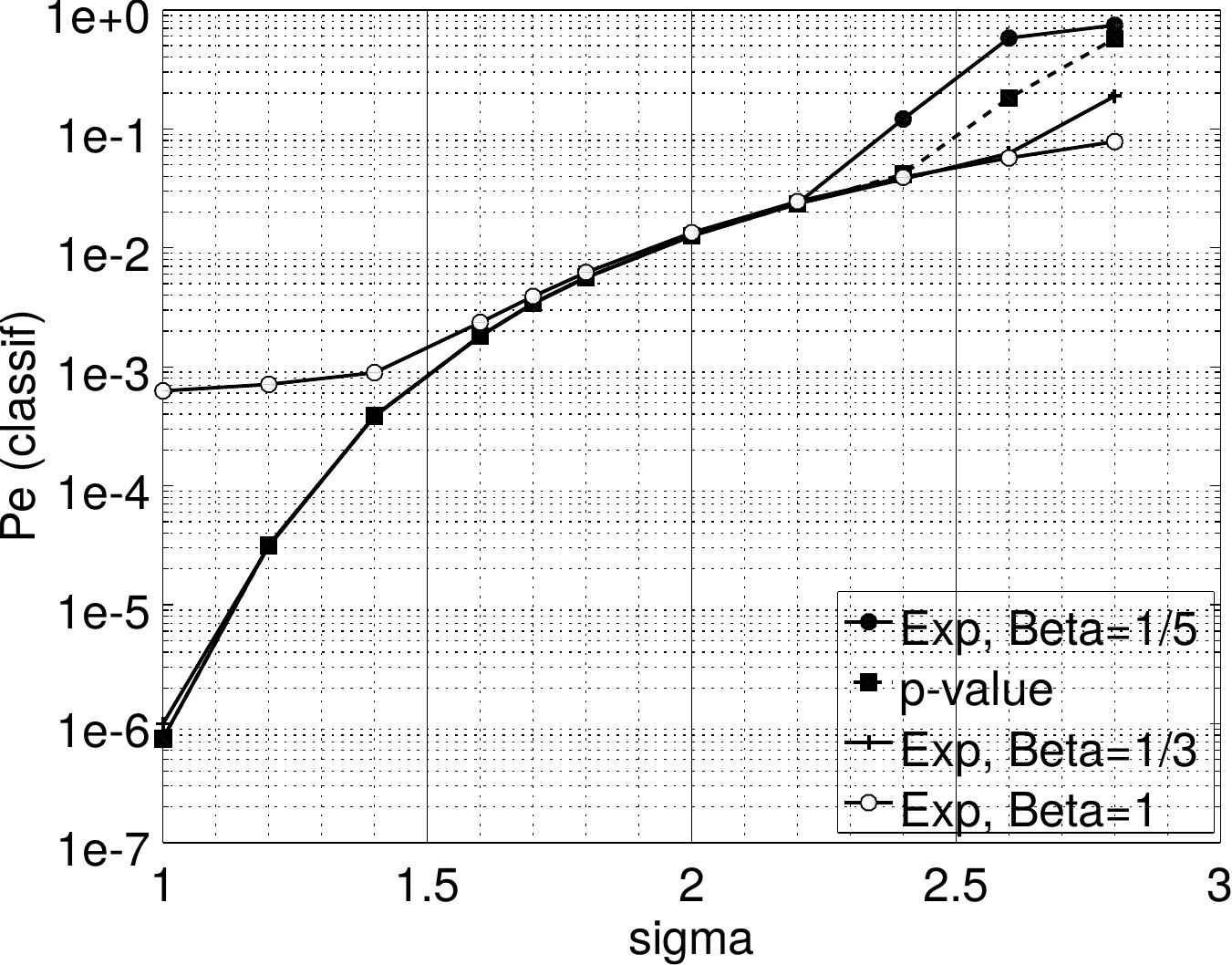}}
  \subfloat[~]{ \includegraphics[width=.48\linewidth]{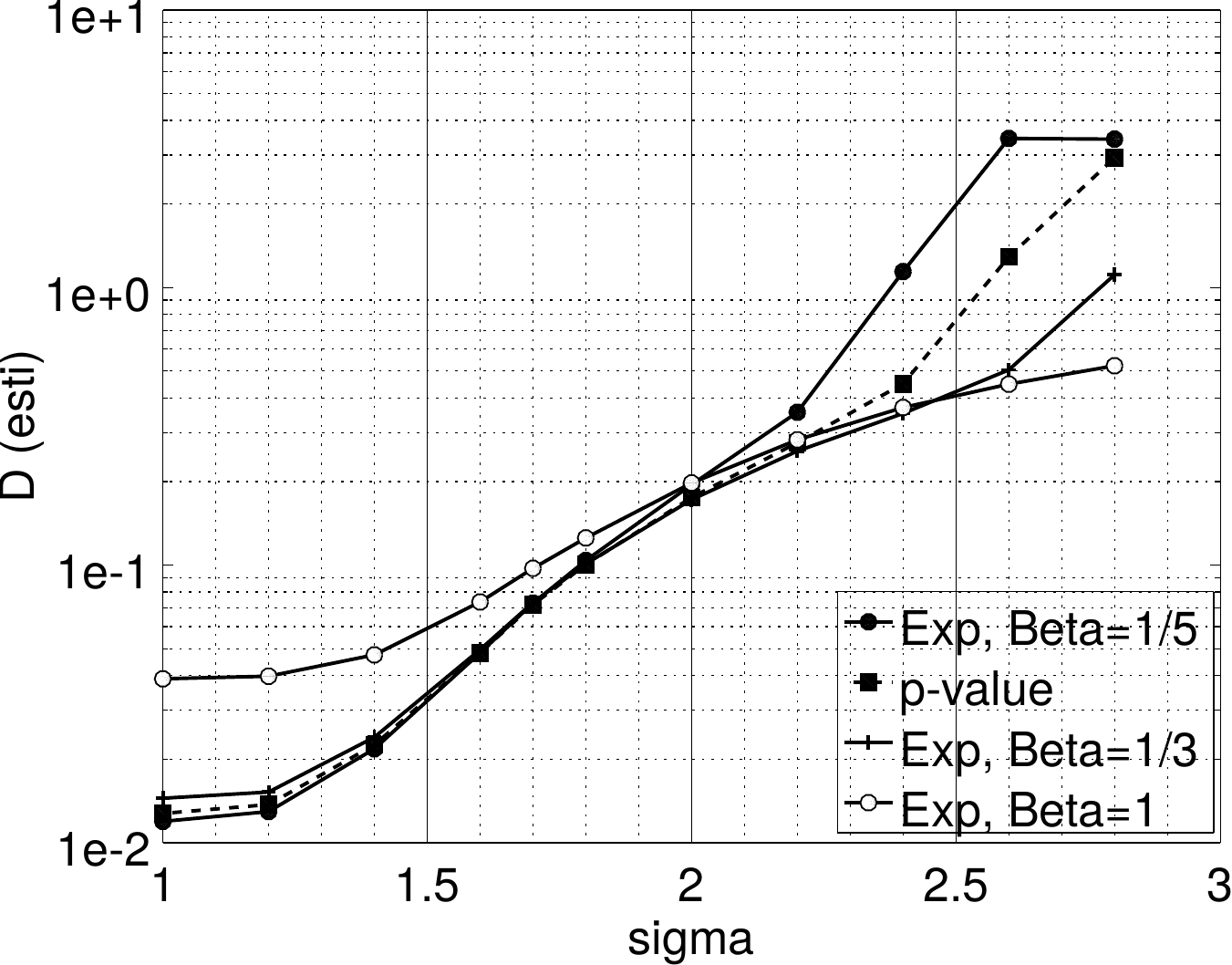}}
\end{center}
\caption{Performance of \algo with Gaussian kernel, with respect to $\sigma$ and for various values of $\beta$: (a) Classification error probability (b) Estimation distortion.}
\label{fig:res_exp}
\end{figure*}

We also evaluated the performance of \algo~with Gaussian kernel for fixed values $\beta = 1$ (as recommended in~\cite{wu02PR}), $\beta=1/3$, $\beta=1/5$.
The results are presented in Figures~\ref{fig:res_exp} (a) and (b).
The parameter $\beta=1$ shows the best performance for high values of $\sigma$, but the worst performance for low values of $\sigma$, and the opposite observation can be made for $\beta=1/5$.
The value $\beta=1/3$ seems to represent the best tradeoff since it it close to the best possible performance for any value of $\sigma$. 
Compared to these results, \algo~with Wald p-value kernel yields good performance unless the value of $\sigma$ becomes too large. 
This result is obtained without having to optimize any parameter, as we have to do for the Gaussian kernel. 
This feature of \algo~\duprazv{with Wald p-value} is useful since it induces no parameter change when considering various dimensions, number of clusters, etc. 
\end{longversion}

\subsubsection{Parameters $\dim = 100$, $K = 10$, $N = 100$ }

\begin{onecol}
\begin{figure*}[t]
\begin{center}
  \subfloat[~]{ \includegraphics[width=.48\linewidth]{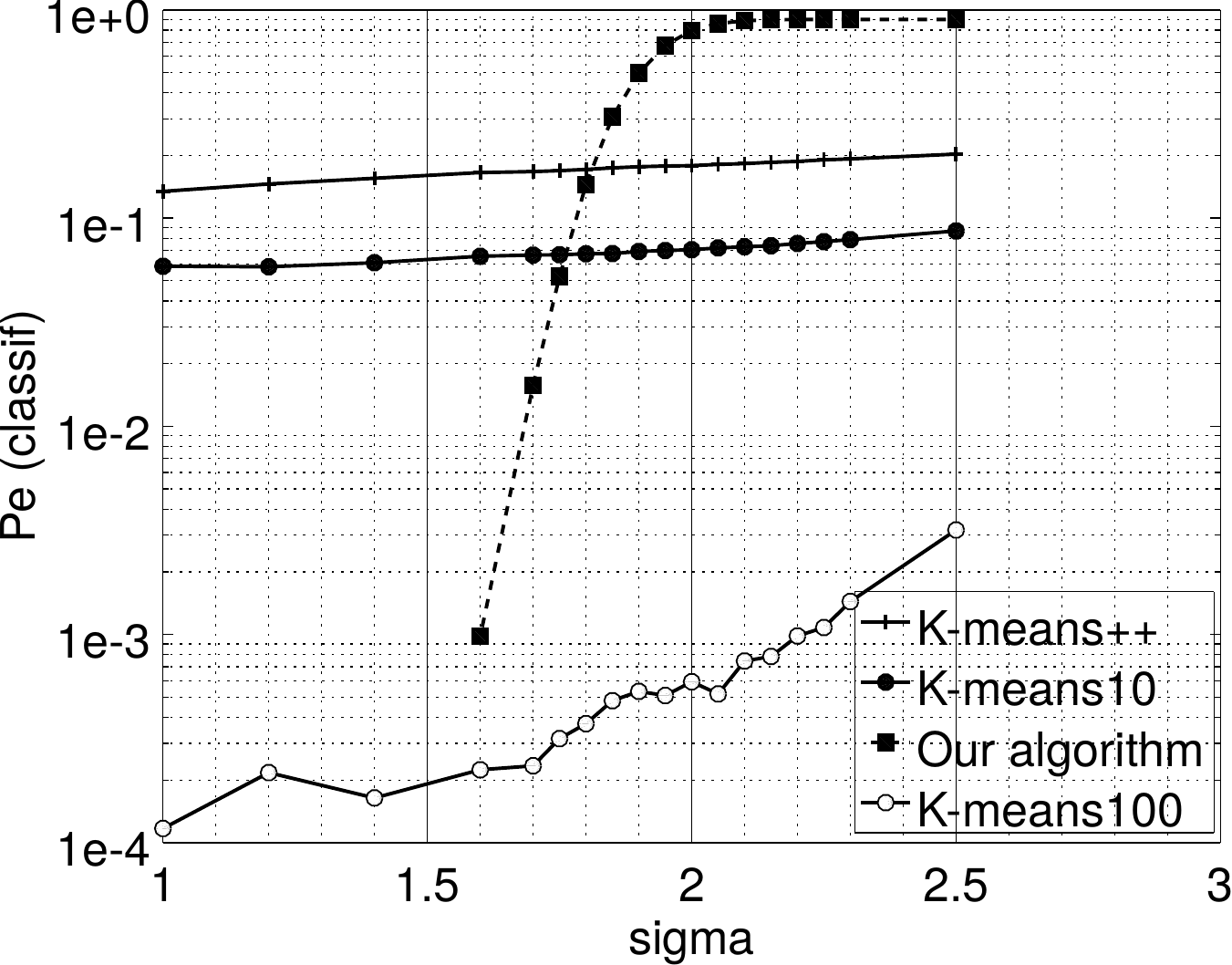}}
  \subfloat[~]{ \includegraphics[width=.48\linewidth]{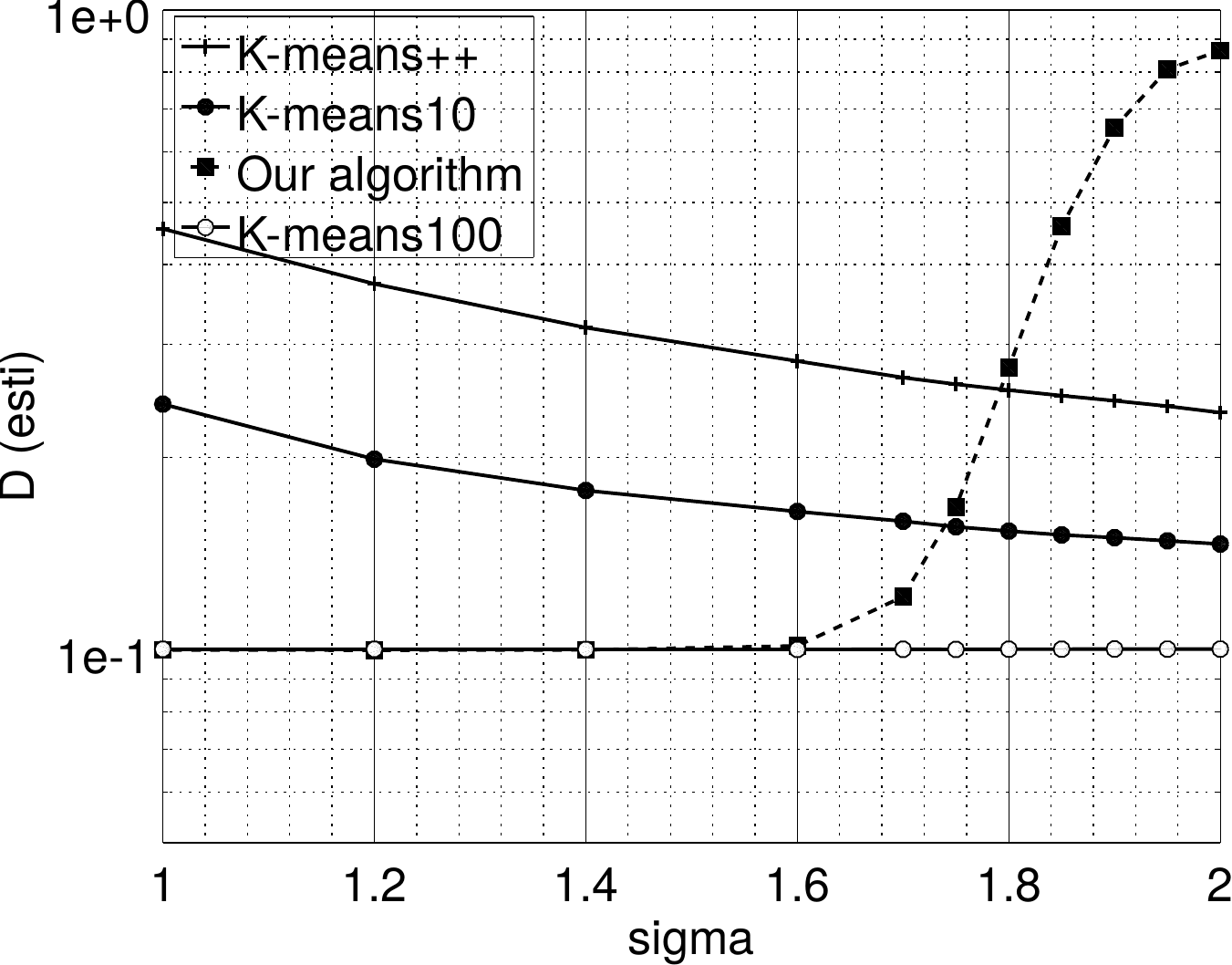}}
    \vspace{-0.4cm}
\end{center}
\caption{Performance evaluation of \algo, K-means10 and K-means100, with $\dim = 100$, $K = 10$, $N = 100$: (a) Classification error probability (b) Estimation distortion. For $\sigma < 1.6$, no value for the classification error probability is reported in subfigure (a) for \algo, since this one commited no classification error over the considered $1000$ realizations. \vspace{-14pt}
}
\label{fig:res_dim100}
\end{figure*}
\end{onecol}

\begin{twocol}
\begin{figure*}[t]
\begin{center}
  \subfloat[~]{ \includegraphics[width=.32\linewidth]{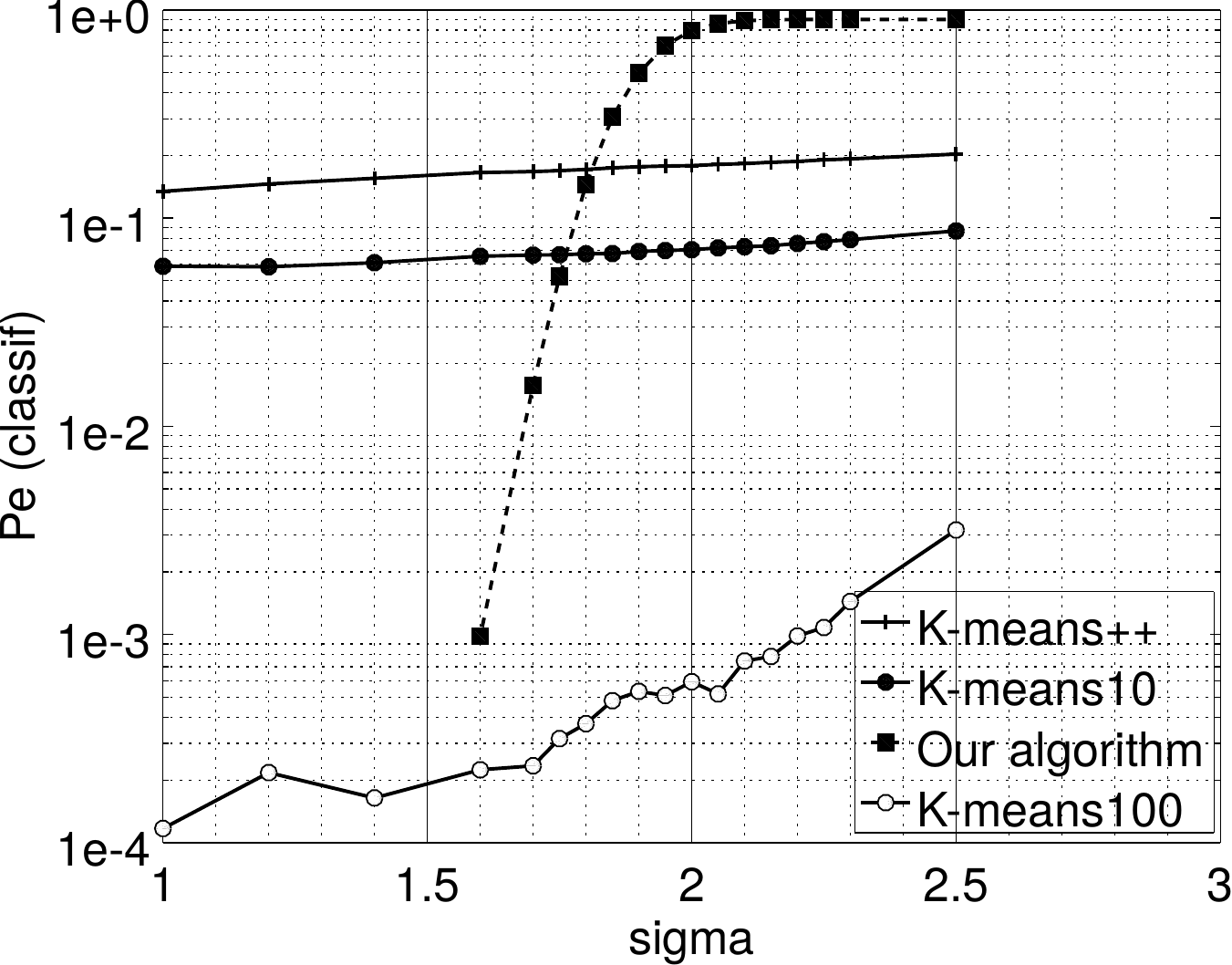}}
  \subfloat[~]{ \includegraphics[width=.32\linewidth]{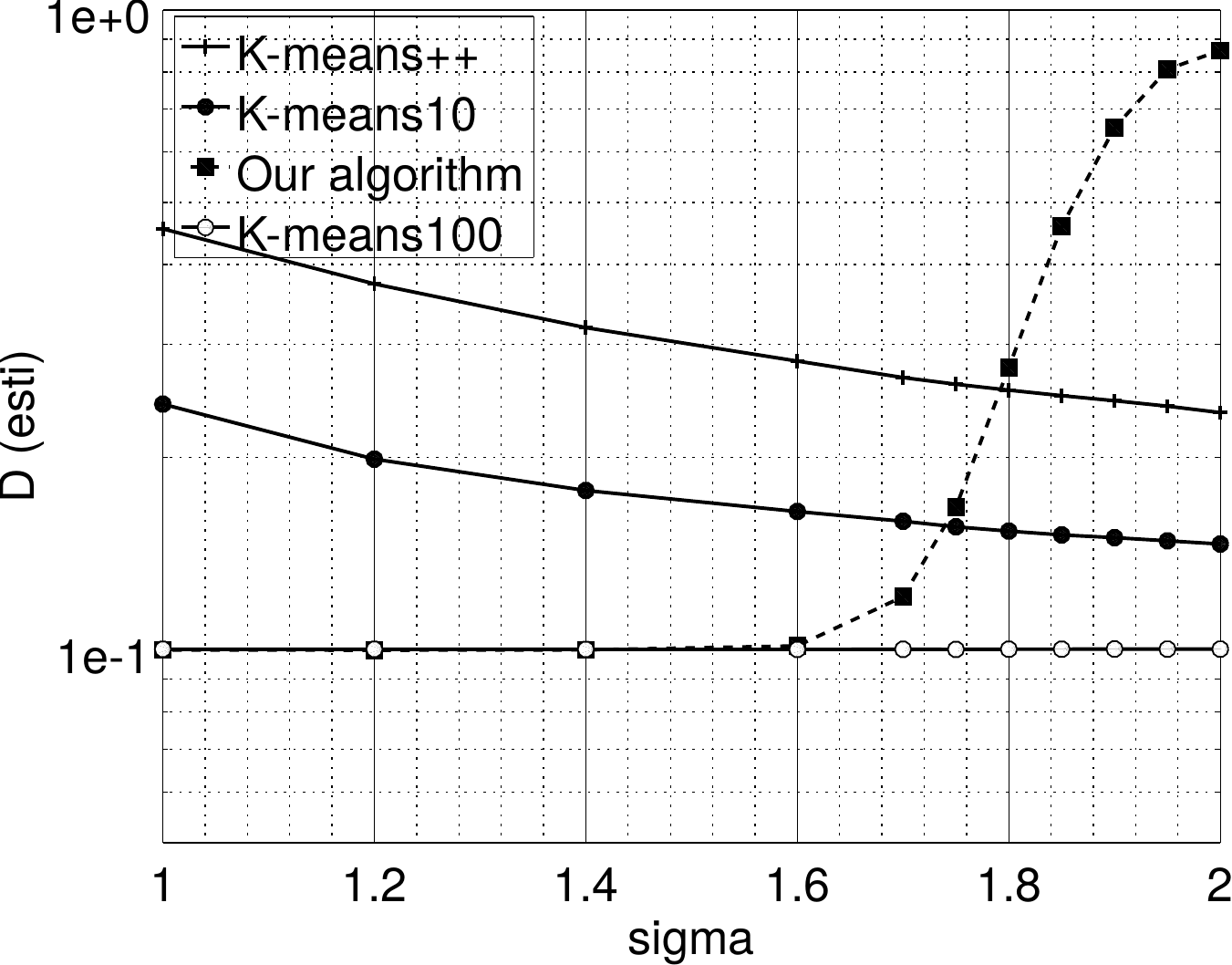}}
\end{center}
\caption{Performance evaluation of \algo, K-means10 and K-means100, with $\dim = 100$, $K = 10$, $N = 100$: (a) Classification error probability (b) Estimation distortion. For $\sigma < 1.6$, no value for the classification error probability is reported in subfigure (a) for \algo, since this one commited no classification error over the considered $1000$ realizations.
}
\label{fig:res_dim100}
\end{figure*}
\end{twocol}

In this part, we consider a higher vector dimension $d=100$, as well as an increased number of clusters $K=10$. The number of observations is set to $N=100$, which gives only $10$ vectors per cluster.
We  still consider various values of $\sigma$, and evaluate the centralized algorithms over $Nt = 10000$ realizations for each considered value of $\sigma$. 
The ten new centroids are generated once for all as $\Ctr_k \sim \mathcal{N}(0,A^2 \Identity{\dim})$ with $A=2$.

As for $\dim=2$, \algo~is benchmarked against K-means++ without replicates, K-means10 and K-means100.
The results are presented in Figures~\ref{fig:res_dim100} (a) and (b).
We first observe that K-means++ and K-means10 replicates perform very poorly with this set of parameters. 
The relatively high number of clusters $(K=10)$ makes it more difficult to obtain a correct initialization with these two solutions, which explains these poor results. 
Regarding our algorithm, it is worth mentioning that for $\sigma<1.6$, the simulations did not return any error in terms of classification over the considered $1000$ realizations.
As a result, \algo~outperforms K-means100 for low values of $\sigma$. 
This is of noticeable importance since the number of data per cluster is relatively small $(N/K=10)$, whereas the theoretical results of Section~\ref{sec:theory} where proved under the conditions that $N$ and the distances between centroids goes to infinity. 
This shows that our algorithm still performs well in non-asymptotic conditions.

On the other hand, our algorithm severely degrades for high values of $\sigma$.
The value $\sigma_{\text{lim}}$ for which our algorithm does not perform good anymore can be determined theoretically as in Section~\ref{subsec:exp_dim2} as $\sigma_{\text{lim}} = 2$.
The region where $\sigma> \sigma_{\text{lim}}$ corresponds again to cases where the clusters are too close to each other for \algo~to be capable of separating them, since it ignores the number of clusters.

In our experiments, K-means was evaluated with replicates and known $K$. In contrast, \algo~uses no replicates and was not provided with the value of $K$. It turns out that, despite these conditions favorable to K-means, \algo~\pastorv{does not incur a too significant performance loss} in comparison to K-means and even outperforms this one as long as the noise variance is not too big. In addition, \algo~can perform so without the need to be repeated several times for proper initialization. For all these reasons, \algo~appears as a good candidate for decentralized clustering.

\begin{longversion}
\subsection{Tests on images}


Additional tests on images were also performed so as to study further the behavior of \algo~in very high dimension. We considered the nine images of Figure \ref{Fig: clean images}. Each of these images has size $256 \times 256$. After computing numerically the minimum distance $\d_{\text{min}}$ between these images, application of (\ref{eq: sigmalim}) with $\level = 10^{-3}$ returns $\sigma_{\text{lim}} =  57$. 

The next experiments are based on the following remark. The {\em a priori} known noise standard deviation $\sigma$ can easily be {\em a posteriori} estimated after clustering. It can then be expected that there should be little difference between the true value of $\sigma$ and its estimate after clustering by \algo, provided that $\sigma$ remains small enough, typically less than $\sigma_{\text{lim}}$. On the opposite, this difference should increase once the noise standard deviation exceeds some value that must be around $\sigma_{\text{lim}}$. 

To verify this claim, we performed clustering by \algo~on a set of noisy images generated as follows. For each $\sigma \in \{40, 41, \ldots, 51, 51, \ldots, 69\}$, we generated $N = 10$ noisy versions of each noiseless image displayed in Figure \ref{Fig: clean images}. By so proceeding, we obtained a set of noisy images. Instances of these noisy images are presented in Figures \ref{Fig: noisy images-sig50} and \ref{Fig: noisy images-sig70} for $\sigma = 50$ and $\sigma = 70$, respectively. After clustering by \algo, we estimated $\sigma$ for comparison to its true value. The value of $\sigma$ was estimated as
\begin{equation}
\widehat{\sigma}^2 = \frac{1}{N d} \sum_{n=1}^N \| \obs_n - \widehat{\ctr}_k (\obs_n) \|^2
\end{equation}
where $\widehat{\ctr}_k (\obs_n) $ is the closest estimated centroid to $\obs_n$. For each value of $\sigma$, we reiterated $100$ times the above process so as to average the noise standard deviation estimates. The obtained average values are displayed in Figure \ref{Fig: est_sigma} \duprazv{with respect to $\sigma$}. 

In this figure, for reasons not yet identified, we observe that $\sigma$ is underestimated (resp. \duprazv{overrestimated}) after clustering by \algo, when $\sigma$ is below (resp. above) $59$. Thus, the value $\sigma_{\text{lim}}$ also characterizes rather well the change in the behavior of the noise standard deviation estimation after clustering by \algo. 

Figure \ref{Fig: est_sigma} also shows that \algo~can itself assess the quality of its clustering and warns the user that the clustering may be failing. Indeed, as soon as the noise standard deviation is overestimated after clustering by \algo, it can be deemed that \algo~is performing out its optimal operating range. To experimentally verify this assertion, we proceeded as follows. For any given tested $\sigma$ and any clustering algorithm, we can always calculate the PSNR of each estimated centroid, with reference to the closest image among the nine of Figure \ref{Fig: clean images}. By so proceeding, we can evaluate the quality of the centroid estimation performed by the clustering. We recall that the PSNR of a given image $I$ with respect to a reference one $I_{\text{ref}}$, both with size $M \times M$, is defined by setting $\text{PSNR} = 10 \log_{10} \left ( d(I)^2/\text{QEM} \right)$ where $d(I)$ is the maximum pixel value of $I$ and $\text{QEM}$ is the quadratic error mean given by $\text{QEM} = \frac{1}{M^2} \sum_{i=1}^M \sum_{j=1}^M \left ( I(i,j) - I_{\text{ref}}(i,j) \right )^2.$

Instances of such PSNRs are given in Table \ref{Tab: PSNRS}. For each $\sigma \in \{50, 60, 70 \}$, the PSNRs were calculated by generating $N = 10$ noisy versions of each clean image of Figure \ref{Fig: clean images}, shuffling the resulting noisy images to form the dataset presented to \algo, K-means10 and K-means100. No averaging were performed to get these values so as to better emphasize the following facts. Clearly, \algo~may fail to find out the correct number of centroids when $\sigma$ becomes too large and, more precisely, above $\sigma_{\text{lim}} = 59$. However, the centroids estimated by \algo~are well estimated, with a PSNR equal to that returned by K-means100. In contrast, K-means10 sometimes fails to estimate correctly the centroids. Indeed, for each $\sigma$ in Table \ref{Tab: PSNRS}, K-means10 returned a centroid mixing several images (see Figure \ref{Fig: mingled centroids}). In short, \algo~may fail to retrieve all the existing centroids; nevertheless, those yielded by \algo~are correctly estimated, whereas K-means may require quite a lot of replicates to perform a correct estimation of these same centroids.

\begin{figure*}[!htbp]
\centering
\vspace*{-0.5cm}
\begin{tabular}{c c c}
\centering
\begin{tabular}{c} 
\small Barbara\\
\hspace{-0.5cm} 
\myinclude{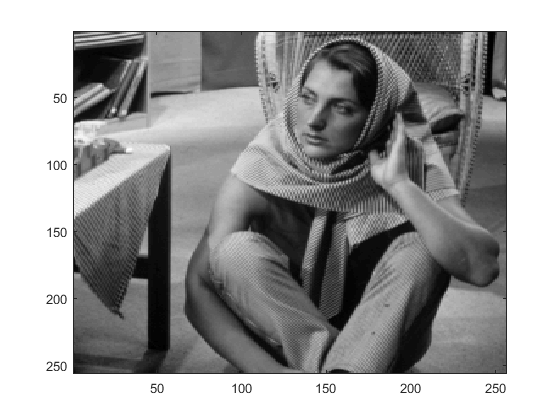}
\end{tabular}
\hspace{-0.5cm} 
&
\begin{tabular}{c} 
\small Cameraman\\
\hspace{-0.5cm} 
\myinclude{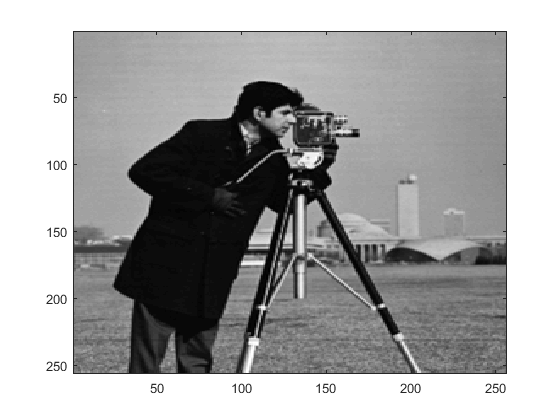}
\end{tabular}
\hspace{-0.5cm} 
&
\begin{tabular}{c} 
\small Einstein\\
\hspace{-0.5cm} 
\myinclude{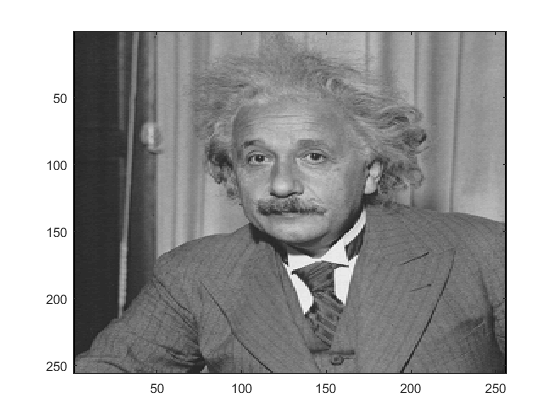}
\end{tabular}
\end{tabular}
\vspace{0.2cm}\\
\begin{tabular}{c c c}
\centering
\begin{tabular}{c} 
\small House\\
\hspace{-0.5cm} 
\myinclude{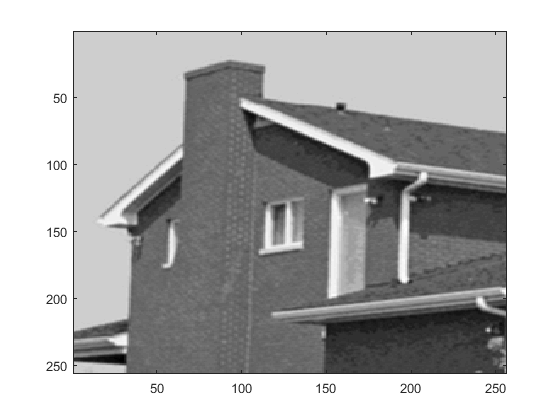}
\end{tabular}
\hspace{-0.5cm} 
&
\begin{tabular}{c} 
\small Jetplane\\
\hspace{-0.5cm} 
\myinclude{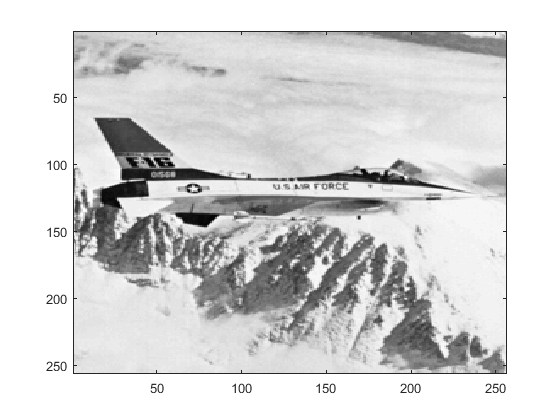}
\end{tabular}
\hspace{-0.5cm} 
&
\begin{tabular}{c} 
\small Lake\\
\hspace{-0.5cm} 
\myinclude{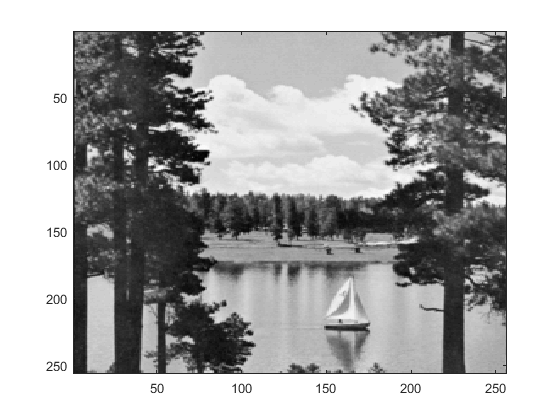}
\end{tabular}
\end{tabular}
\vspace{0.2cm}\\
\begin{tabular}{c c c}
\centering
\begin{tabular}{c} 
\small Lena\\
\hspace{-0.5cm} 
\myinclude{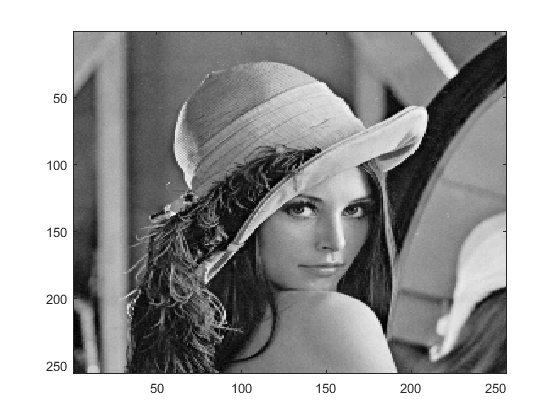}
\end{tabular}
\hspace{-0.5cm} 
&
\begin{tabular}{c} 
\small Mandrill\\
\hspace{-0.5cm} 
\myinclude{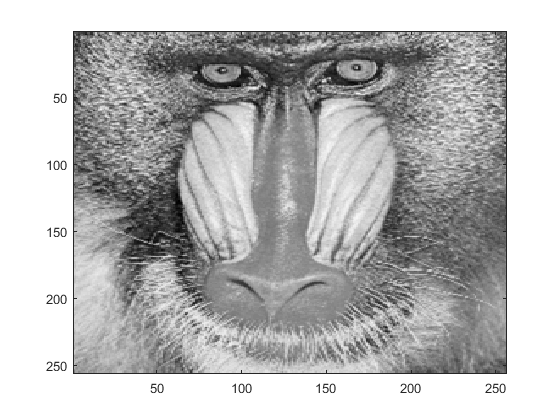}
\end{tabular}
\hspace{-0.5cm} 
&
\begin{tabular}{c} 
\small Peppers \\
\hspace{-0.5cm} 
\myinclude{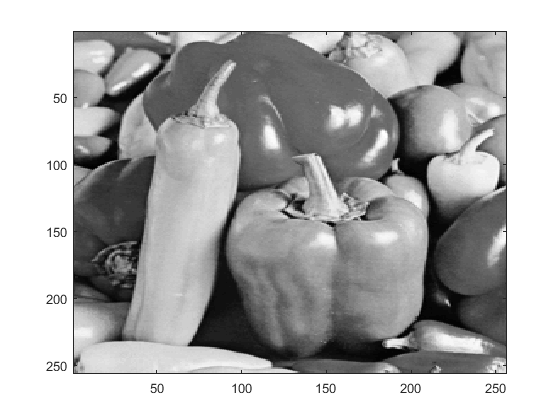}
\end{tabular}
\end{tabular}
\caption{}
\label{Fig: clean images}
\end{figure*}

\begin{figure*}[!htbp]
\centering
\vspace*{-0.5cm}
\begin{tabular}{c c c}
\centering
\begin{tabular}{c} 
\small Barbara \vspace{-0.25cm} \\ ($\sigma = 50$)\\
\hspace{-0.5cm} 
\myinclude{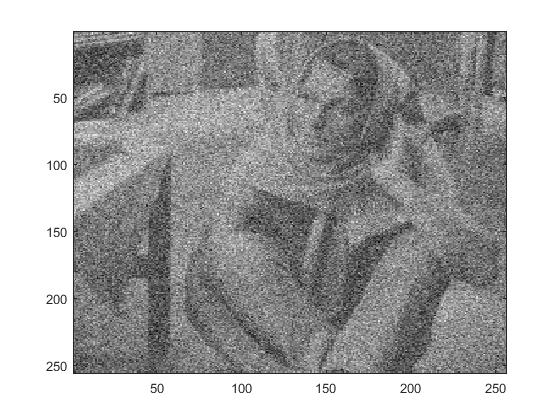}
\end{tabular}
\hspace{-0.5cm} 
&
\begin{tabular}{c} 
\small Cameraman \vspace{-0.25cm} \\ ($\sigma = 50$)\\
\hspace{-0.5cm} 
\myinclude{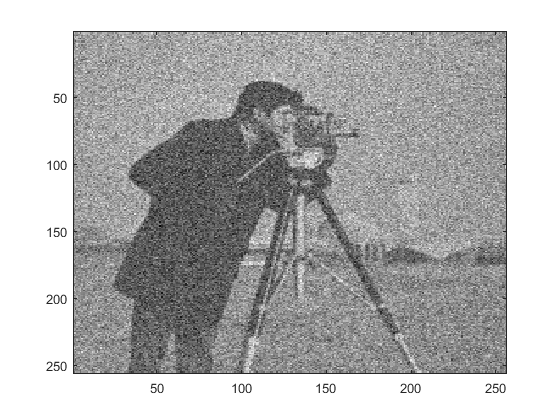}
\end{tabular}
\hspace{-0.5cm} 
&
\begin{tabular}{c} 
\small Einstein \vspace{-0.25cm} \\ ($\sigma = 50$)\\
\hspace{-0.5cm} 
\myinclude{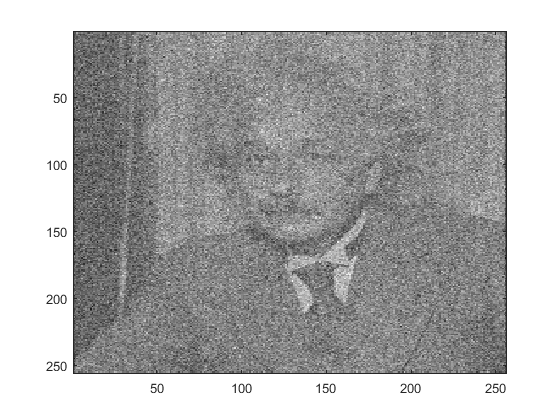}
\end{tabular}
\end{tabular}
\vspace{0.2cm}\\
\begin{tabular}{c c c}
\centering
\begin{tabular}{c} 
\small House \vspace{-0.25cm} \\ ($\sigma = 50$)\\
\hspace{-0.5cm} 
\myinclude{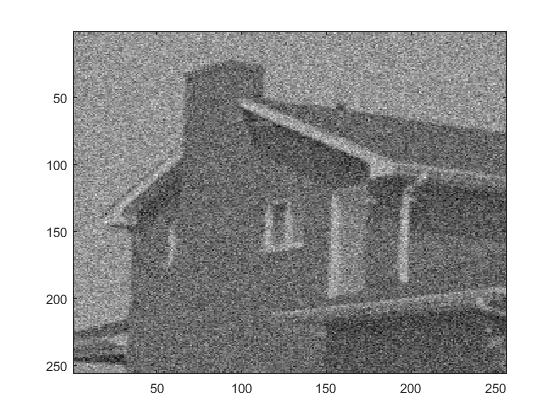}
\end{tabular}
\hspace{-0.5cm} 
&
\begin{tabular}{c} 
\small Jetplane \vspace{-0.25cm} \\ ($\sigma = 50$) \\
\hspace{-0.5cm} 
\myinclude{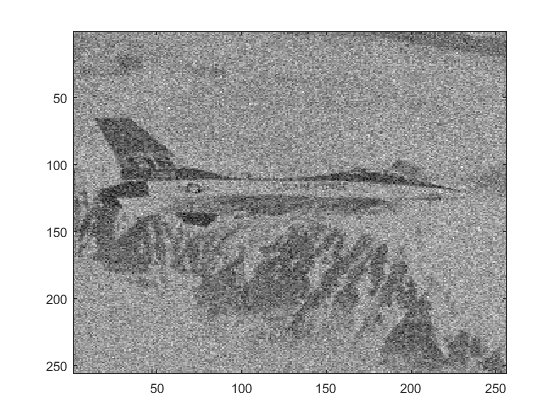}
\end{tabular}
\hspace{-0.5cm} 
&
\begin{tabular}{c} 
\small Lake \vspace{-0.25cm} \\ ($\sigma = 50$) \\
\hspace{-0.5cm} 
\myinclude{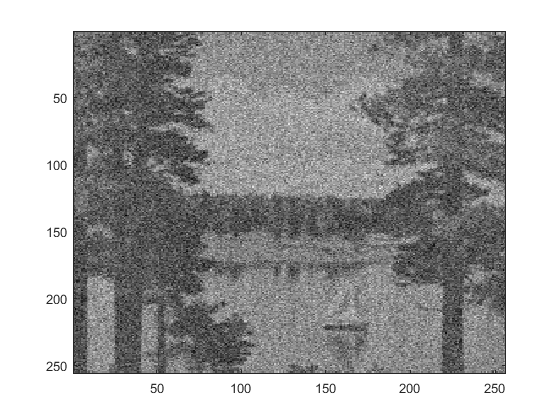}
\end{tabular}
\end{tabular}
\vspace{0.2cm}\\
\begin{tabular}{c c c}
\centering
\begin{tabular}{c} 
\small Lena \vspace{-0.25cm} \\ ($\sigma = 50$) \\
\hspace{-0.5cm} 
\myinclude{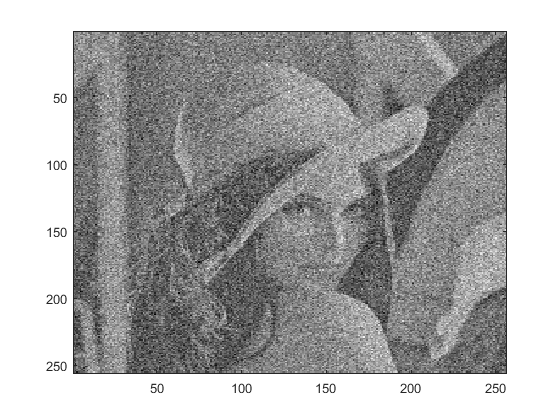}
\end{tabular}
\hspace{-0.5cm} 
&
\begin{tabular}{c} 
\small Mandrill \vspace{-0.25cm} \\ ($\sigma = 50$) \\
\hspace{-0.5cm} 
\myinclude{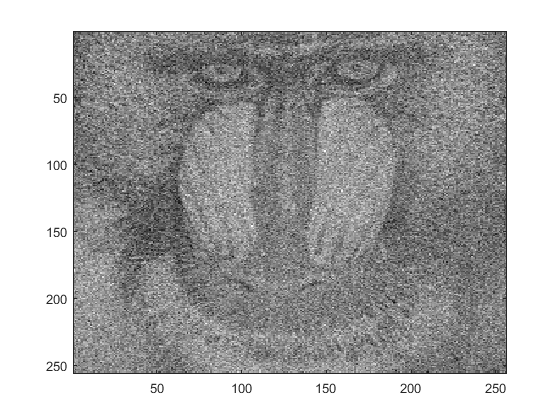}
\end{tabular}
\hspace{-0.5cm} 
&
\begin{tabular}{c} 
\small Peppers \vspace{-0.25cm} \\ ($\sigma = 50$) \\
\hspace{-0.5cm} 
\myinclude{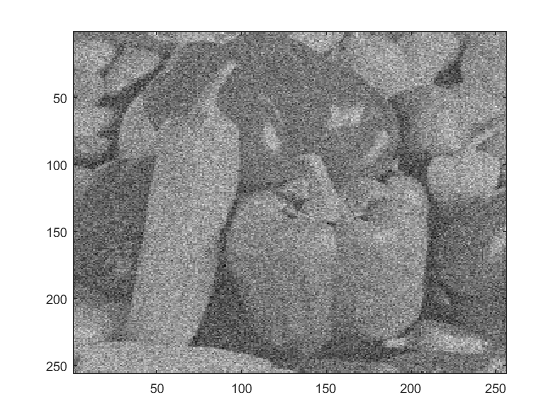}
\end{tabular}
\end{tabular}
\caption{}
\label{Fig: noisy images-sig50}
\end{figure*}

\begin{figure*}[!hbp]
\centering
\vspace*{-0.5cm}
\begin{tabular}{c c c}
\centering
\begin{tabular}{c} 
\small Barbara \vspace{-0.25cm} \\ ($\sigma = 70$)\\
\hspace{-0.5cm} 
\myinclude{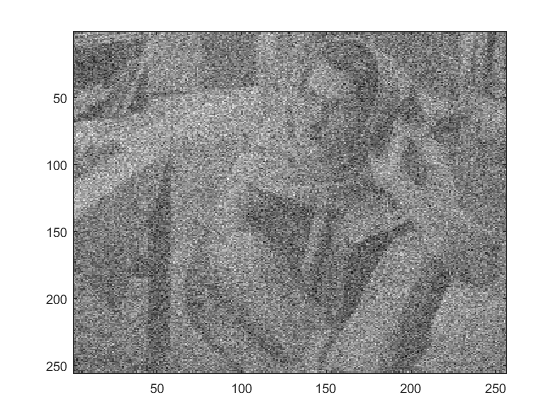}
\end{tabular}
\hspace{-0.5cm} 
&
\begin{tabular}{c} 
\small Cameraman \vspace{-0.25cm} \\ ($\sigma = 70$)\\
\hspace{-0.5cm} 
\myinclude{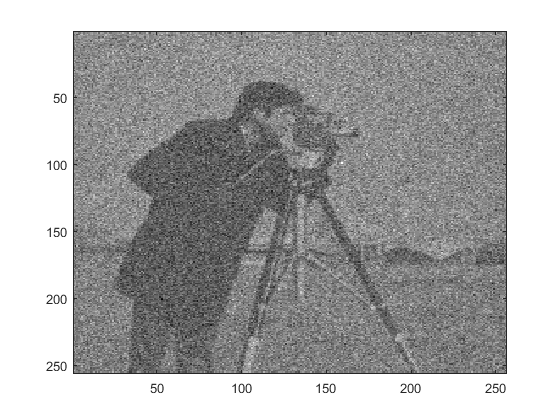}
\end{tabular}
\hspace{-0.5cm} 
&
\begin{tabular}{c} 
\small Einstein \vspace{-0.25cm} \\ ($\sigma = 70$)\\
\hspace{-0.5cm} 
\myinclude{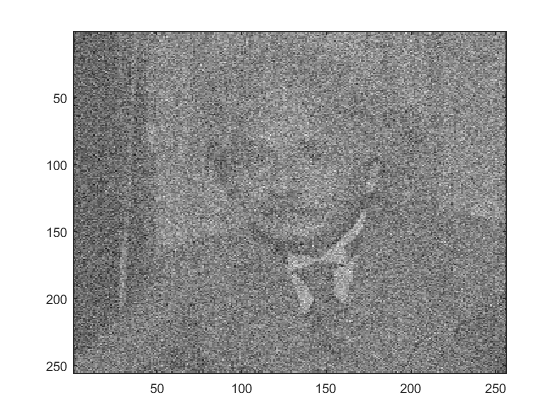}
\end{tabular}
\end{tabular}
\vspace{0.2cm}\\
\begin{tabular}{c c c}
\centering
\begin{tabular}{c} 
\small House \vspace{-0.25cm} \\ ($\sigma = 70$)\\
\hspace{-0.5cm} 
\myinclude{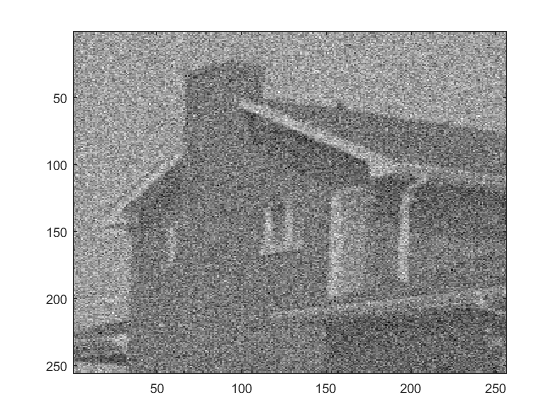}
\end{tabular}
\hspace{-0.5cm} 
&
\begin{tabular}{c} 
\small Jetplane \vspace{-0.25cm} \\ ($\sigma = 70$) \\
\hspace{-0.5cm} 
\myinclude{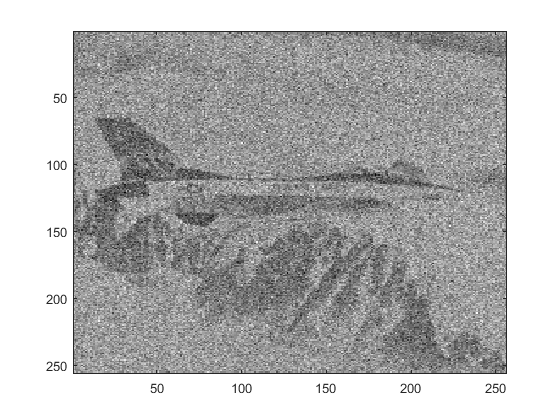}
\end{tabular}
\hspace{-0.5cm} 
&
\begin{tabular}{c} 
\small Lake \vspace{-0.25cm} \\ ($\sigma = 70$) \\
\hspace{-0.5cm} 
\myinclude{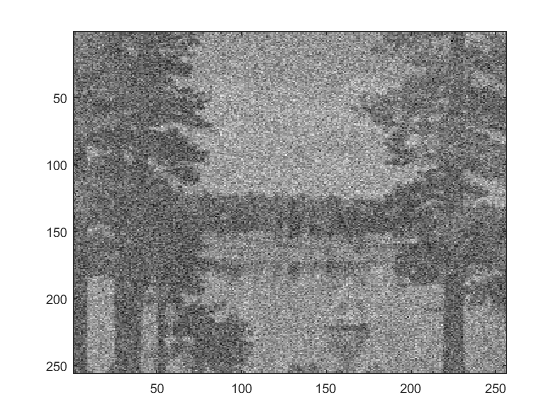}
\end{tabular}
\end{tabular}
\vspace{0.2cm}\\
\begin{tabular}{c c c}
\centering
\begin{tabular}{c} 
\small Lena \vspace{-0.25cm} \\ ($\sigma = 70$) \\
\hspace{-0.5cm} 
\myinclude{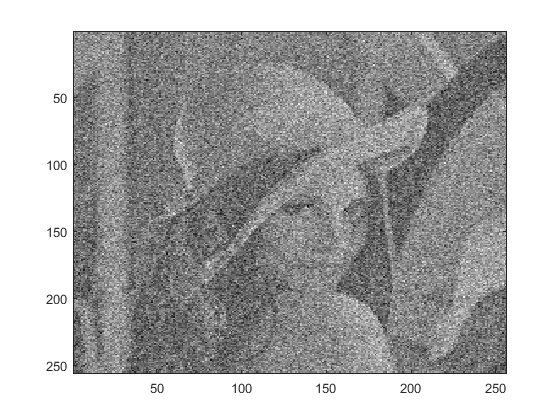}
\end{tabular}
\hspace{-0.5cm} 
&
\begin{tabular}{c} 
\small Mandrill \vspace{-0.25cm} \\ ($\sigma = 70$) \\
\hspace{-0.5cm} 
\myinclude{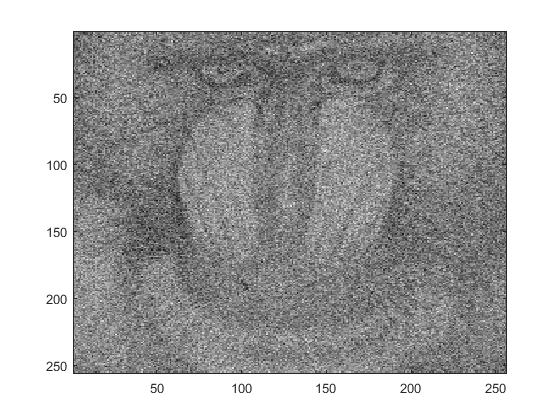}
\end{tabular}
\hspace{-0.5cm} 
&
\begin{tabular}{c} 
\small Peppers \vspace{-0.25cm} \\ ($\sigma = 70$) \\
\hspace{-0.5cm} 
\myinclude{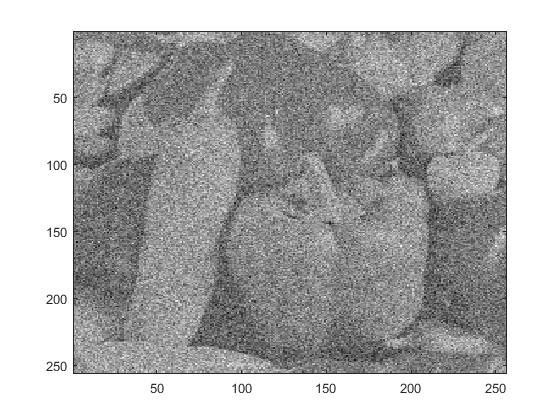}
\end{tabular}
\end{tabular}
\caption{}
\label{Fig: noisy images-sig70}
\end{figure*}

\begin{figure*}[!htbp]
\centering
\vspace*{-0.5cm}
\begin{tabular}{c c c}
\centering
\begin{tabular}{c} 
\small $\sigma = 50$ \\
\hspace{-0.5cm} 
\myinclude{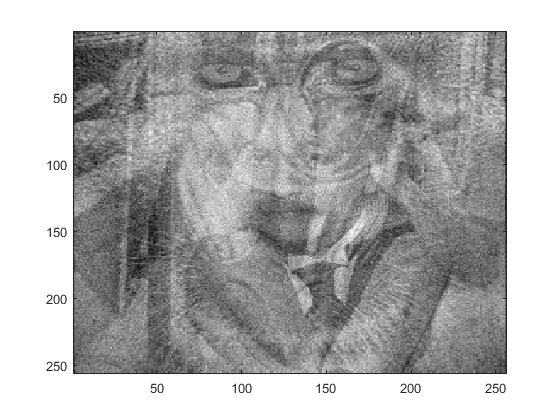}
\end{tabular}
\hspace{-0.5cm} 
&
\begin{tabular}{c} 
\small $\sigma = 60$ \\
\hspace{-0.5cm} 
\myinclude{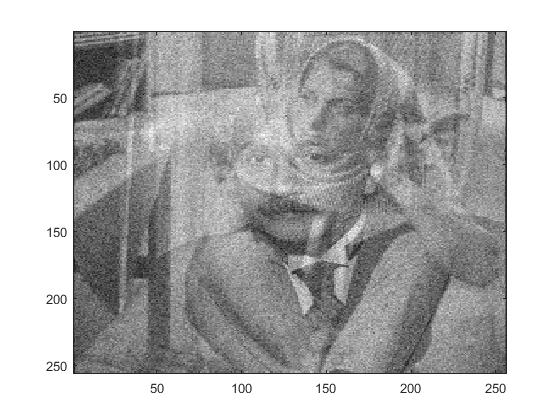}
\end{tabular}
\hspace{-0.5cm} 
&
\begin{tabular}{c} 
\small $\sigma = 70$ \\
\hspace{-0.5cm} 
\myinclude{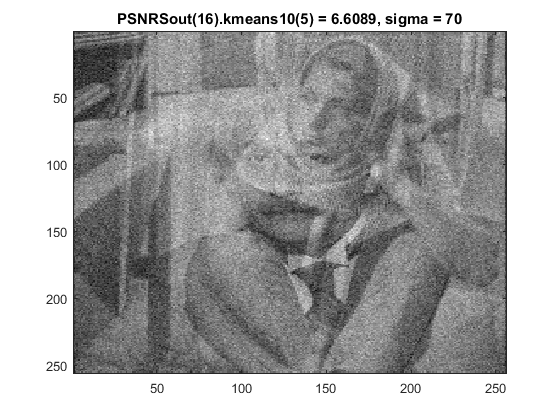}
\end{tabular}
\end{tabular}
\caption{Examples of centroids estimated by K-means10 and resulting from the mixing of several images. The PSNRs of such centroids are respectively $3.39$, $5.54$ and $6.61$ (see Table \ref{Tab: PSNRS})}
\label{Fig: mingled centroids}
\end{figure*}

\begin{figure*}[!htbp]
\centering
\includegraphics[width=7.5cm, height=6.0cm]{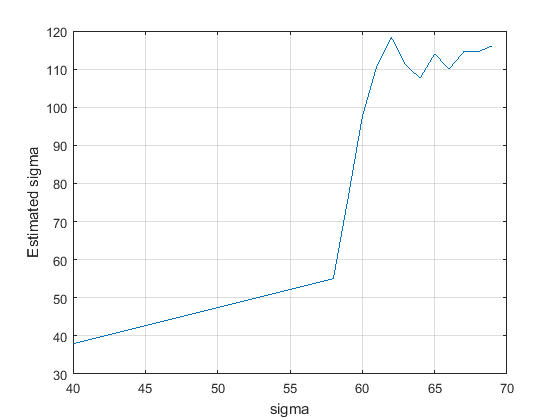}
\caption{Estimate of the noise standard deviation calculated after clustering vs. true value of this noise standard deviation. For $\sigma \in \{40, 41, \ldots, 51, 51, \ldots, 69\}$, $N = 10$ noisy versions of each clean image of Figure \ref{Fig: clean images} are generated. The resulting images, after shuffling, are presented to \algo~for clustering and estimation of the noise standard deviation. Below $\sigma_{\text{lim}}$, the estimate is always less than the actual value. Beyond $\sigma_{\text{lim}}$, the estimate becomes to increase significantly. This increase could be used by the algorithm to assess itself the relevance of its clustering.}
\label{Fig: est_sigma}
\end{figure*}

\begin{table}[!]
\caption{PSNRs of the centroids estimated by \algo, $K$-mean10 and K-means100, for various noise standard deviations. A cross at a junction between a colum and a row indicates that the image to the left is not the closest one to any of the centroids estimated by the algorithm specified by the column. We then say that this algorithm did not retrieve the image. Two values in the same case indicate that the image naming the row was the closest one to two estimated centroids. For instance, when $\sigma = 50$, K-means10 returned $2$ estimated centroids, to which the closest clean image was `Lake'.  K-means10 also yields $2$ other estimates for which the closest clean image was `Lena'. Consequently, two clean images, namely `Barbara' and `Mandrill', were never found to be the closest ones to centroids estimated by K-means10. Regarding \algo, it may fail to retrieve all the centroids. For instance, `Barbara' was not retrieved by \algo~among the estimated centroids. However, the PSNRs of the images actually retrieved by \algo~are close to those yielded by K-means100 for $\sigma = 50$. Finally, values in boldface correspond to estimated centroids that appear as noisy mixtures of several other images (see Figure \ref{Fig: mingled centroids}). These PSNRs are significantly smaller than the other ones.}
\tiny
\begin{tabular}{ | c || c | c | c || c | c | c || c | c | c || c | c | c |}
\cline{1-10}
\hline
            & \multicolumn{3}{ c || }{PSNR ($\sigma = 50$) } 
            & \multicolumn{3}{ c || }{PSNR ($\sigma = 60$) } 
            & \multicolumn{3}{ c || }{PSNR ($\sigma = 70$) } \\
\hline
            & \algo  & K-means10  & K-means100 
            & \algo & K-means10  & K-means100 
            & \algo & K-means10  & K-means100             
            \\            
\hline
\hline
Barbara       & 10 & X & 10.01 & X & \textbf{5.54} & 9.96 & X & \textbf{6.61} & 10 \\
\hline
Cameraman  & 9.98  & 9.98  & 9.98 & 9.98 & 9.98 & 9.98 & 10 & 10 & 9.99 \\
\hline
Einstein       & 10.04  & \textbf{3.39} & 10.04 & X & X & 10.01 & X & X & 10\\
\hline
House         & 10.01  & 10.02  & 10.02  & 10.01 & 5.99/7.77 & 10.01 & 10.04 & 6.97/7.01 & 10.03 \\
\hline
Jetplane      &  10.02 & 10.02  & 10.02  & 9.96 & 9.96 & 9.96 & 10.04 & 10.04 & 10.04\\
\hline
Lake           &  9.98 & 8.47/4.76 & 9.98  & 9.98 & 9.98 & 9.98 & 10.00 & 10.00 & 10.00 \\
\hline
Lena           & 10.03  & 6/7.78 & 10.03  & 9.98 & 9.98 & 9.98 & X & 10.00 & 10.00 \\
\hline
Mandrill      & 10.01  & X & 10.01  & 10.03  & 10.03 & 10.03 & X & 9.96 & 9.96\\
\hline
Peppers      & 10.02 & 10.02 & 10.02 & 9.96 & 9.96 & 9.96 & X & 10.02 & 10.00 \\
\hline
\end{tabular}
\label{Tab: PSNRS}
\end{table}


\end{longversion}

\begin{shortversion}
\subsubsection{Other simulation results}
We also obtained other simulation results we do not present in this version for space reasons.
We first evaluated the performance of \algo~where the Wald p-value kernel is replaced with the Gaussian kernel (see Section~\ref{sec:mestimator}). Although the Gaussian kernel can yield performance similar to Wald p-value kernel, its main drawback is that the parameter $\beta$ must be manually adjusted for each dimension and each $\sigma$, which is not desirable in a distributed setting. Second, we applied \algo~with the Wald p-value kernel to noisy images. The high dimension $256 \times 256$ of the images do not affect the performance of our algorithm. In addition, we observed that \algo~still works in far from asymptotic conditions with $9$ clusters or images and $10$ images per cluster.
For more details, see the extended version of the paper on ArXiv. 
\end{shortversion}

\subsection{Decentralized algorithm}

\begin{onecol}
\begin{figure*}[t]
\begin{center}
  \subfloat[~]{ \includegraphics[width=.48\linewidth]{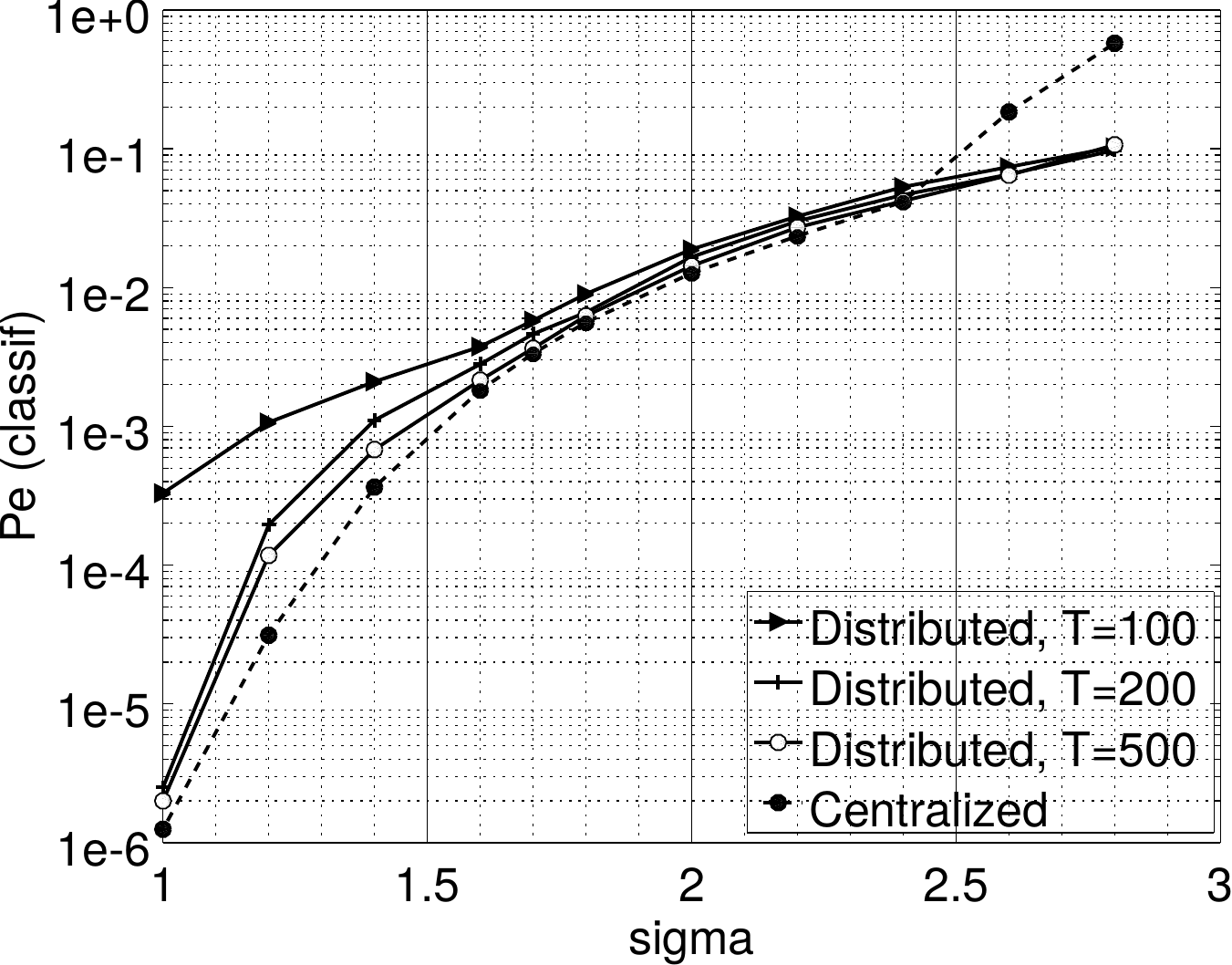}}
  \subfloat[~]{ \includegraphics[width=.48\linewidth]{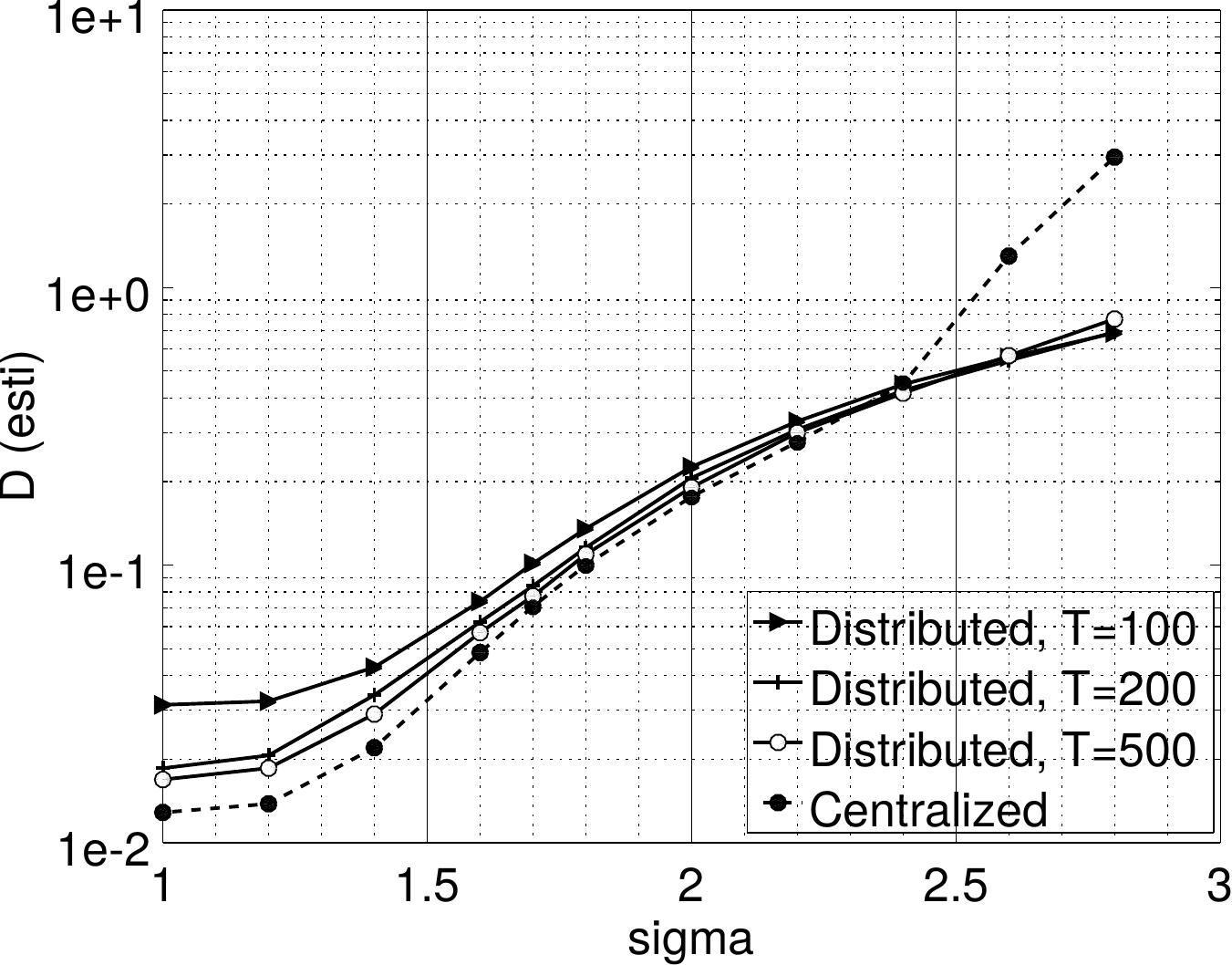}}
  \vspace{-0.2cm}
\end{center}
\caption{Performance evaluation of \dalgo~for various values of $T$, with $\dim = 2$, $K = 4$, $N = 400$: (a) Classification error probability (b) Estimation distortion.\vspace{-12pt}}
\label{fig:res_distrib}
\end{figure*}
\end{onecol}

\begin{twocol}
\begin{figure*}[t]
\begin{center}
  \subfloat[~]{ \includegraphics[width=.32\linewidth]{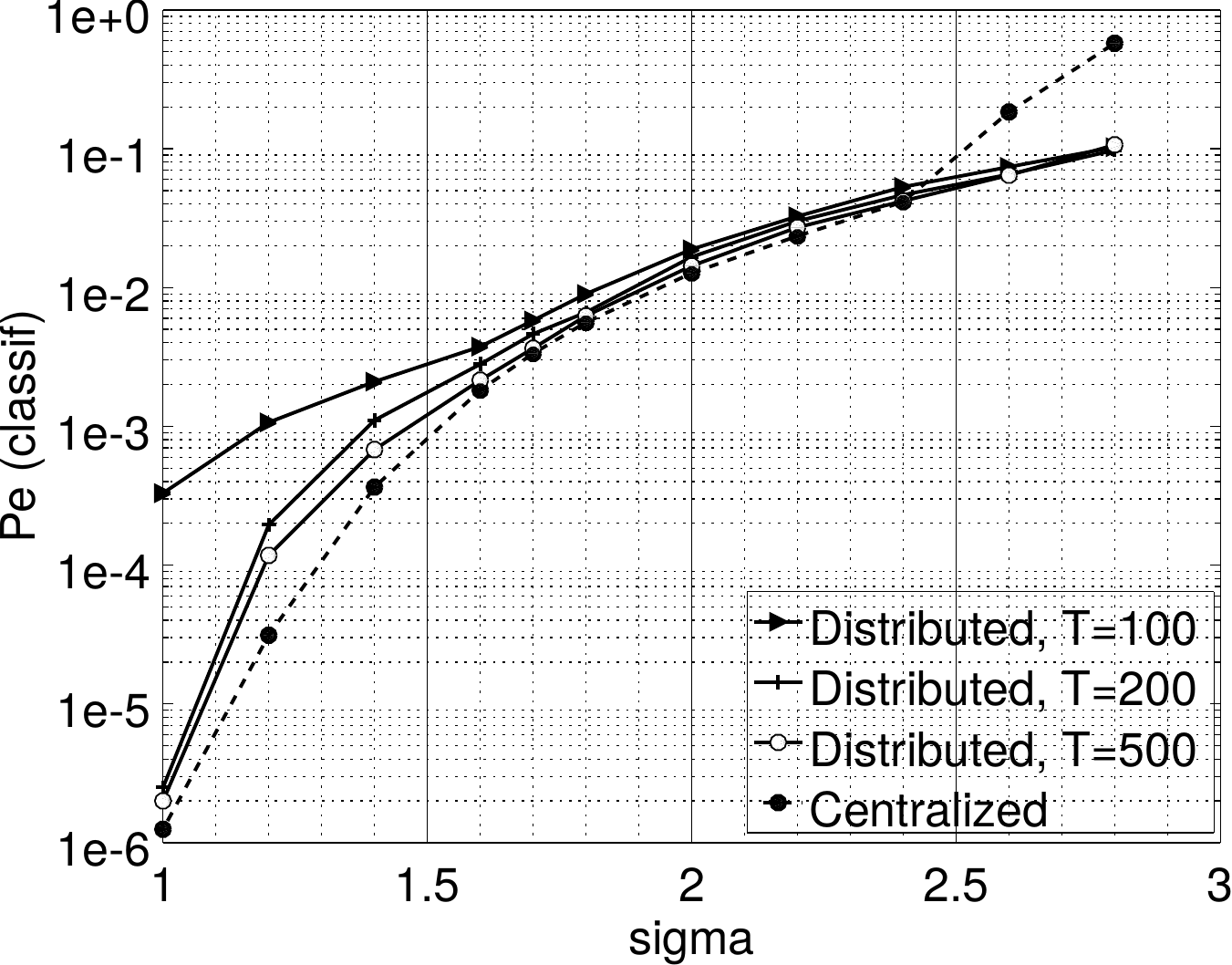}}
  \subfloat[~]{ \includegraphics[width=.32\linewidth]{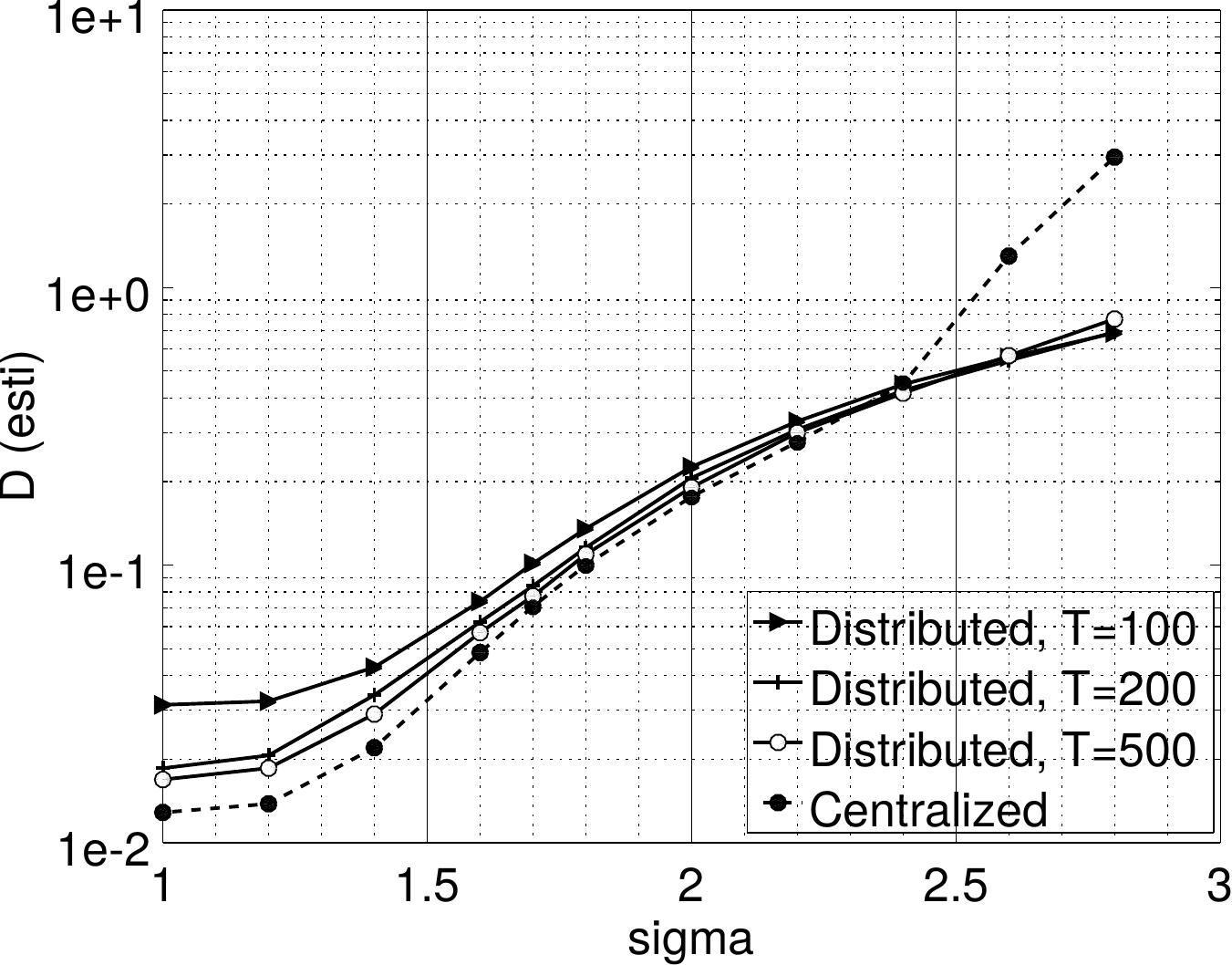}}
  \vspace{-0.2cm}
\end{center}
\caption{Performance evaluation of \dalgo, with $\dim = 2$, $K = 4$, $N = 400$: (a) Classification error probability (b) Estimation distortion.\vspace{-12pt}}
\label{fig:res_distrib}
\end{figure*}
\end{twocol}

We now evaluate the performance of the decentralized version \dalgo~of our algorithm.
In this part, we consider again the parameters $\dim = 2$, $K = 4$, $N = 400$, and the data are generated as in Section~\ref{subsec:exp_dim2}.
We address the influence of the number of the time slots $T$ allocated to the estimation of each centroid. We proceed by comparing, when this parameter varies, the performance of \dalgo~to that of \algo. In all the considered cases, the parameter $L$, which represents the number of received partial sums before updating the centroid is chosen as $L=T/10$. Of course, in all the subsequent experiments, M-estimation of the centroids is performed by using the Wald p-value kernel. The results are presented in Figures~\ref{fig:res_distrib} (a) and (b). 
We see that it is possible to choose a value of $T$ such as the decentralized algorithm undergoes only a limited performance degradation compared to the centralized algorithm. 
A small value $T=100$ induces a rather important performance loss, whereas higher values $T=300$ and $T=500$ provide almost the same performance as \algo, both in terms of classification error probability and estimation distortion. These two figures hence permit to conclude that \dalgo~performs almost as well as \algo, which means that \dalgo~is also competive compared to the centralized versions of K-means, without suffering from the same drawbacks (no need to estimate the number of clusters, no need for replicates).

\section{Conclusion \& perspectives}\label{sec:conclusion}
In this paper, we have introduced a clustering algorithm for statistical signal processing applications. This algorithm does not need to know the number of clusters and is less sensitive to initialization that the K-means algorithm.
These properties make our algorithm suitable for decentralized clustering and this is why we also proposed a distributed version of the algorithm.
It is worth emphasizing that all the steps of our algorithm were theoretically derived and analyzed.  
Both the theoretical analysis and the simulations assess the efficiency of the centralized and decentralized algorithms. 

From a more general point of view, we have introduced a new methodology for clustering. This methodology relies on a statistical model of the measurements. In this methodology, clustering is performed via M-estimation with score function derived from the p-value of a Wald test that is optimal for the considered model. This methodology can be adapted to other signal models, which may allow for adressing more general clustering problems that standard algorithms such as K-means can hardly handle. For instance, we could consider heterogeneous sensors that would collect measurement vectors with different variances. We might also introduce a more general notion of cluster, of which centroids would be random with bounded variations less than a value $\tol \geqslant 0$ in norm (this definition corresponds to $\tau\geqslant 0$ in the computations in Appendices).

\section*{Appendices}

\appendices
Given $\tol \in [0,\infty)$, let $\tsub(\tol)$ be the unique real value such that $\Marcum(\tol,\tsub(\tol)) = \level$. In particular, $\TheThreshold{} = \tsub(0)$. The results \pastorv{stated in} Appendices \ref{Sec: RDT pvalue},~\ref{App: robustness} and \ref{App: asymptotic behaviors} below involve $\tsub(\tol)$ and $w = \Marcum( \tol, \bullet)$ instead of merely $\TheThreshold{}$ and $w = \Marcum( 0, \bullet)$, respectively. The application of these results in the main core of the paper thus concerns the particular case $\tol = 0$. The reason why we present these results for any $\tol \in [0,\infty)$ is twofold. First, the proof of these results in the particular case where $\tol = 0$ would not be significantly simpler than that for the general case. Second, having results that hold for any $\tol \in [0,\infty)$ opens prospects described in the concluding section of the paper.

\section{p-value of Wald test}
\label{Sec: RDT pvalue}
\pastorv{For reasons evoked just above, we consider a more general case than that actually needed in the paper.} As a preliminary result, we need the following lemma.
\begin{Lemma}
\label{Lemma: tsub decreases}
Given $\pastorv{\tol} \in [ \, 0 \, , \, \infty \, )$, the map $\level \in ( \, 0 \, , \, 1 \, ] \mapsto \tsub(\pastorv{\tol}) \in [ \, 0 \, , \infty \, )$ is strictly decreasing.
\end{Lemma}
\begin{proof}
Let $\radius$ be some element of $[0,\infty)$ and consider two elements $\level$ and $\level'$ of $(0,1]$. We have $\Marcum(\radius , \tsub(\radius)) = \level$ and $\Marcum(\radius , \tsubprime(\radius)) = \level'$. If $\level < \level'$, we thus have $\Marcum(\radius , \tsub(\radius)) < \Marcum(\radius , \tsubprime(\radius))$, which implies that $\tsub(\radius) > \tsubprime(\radius)$ since $\Marcum(\radius , \mycdot)$ is strictly decreasing.
\end{proof}
\medskip
\color{black}
\pastorv{Let $\tol \in [0,\infty)$. Given $\ybm \in \Rset^\dim$, set $\level' = \Marcum(\tol / \sigma_0, \pastorv{\| \ybm \| / \sigma_0})$.}
\pastorv{Since $\level' \in (0,1)$,
we have $$\Marcum(\tol / \sigma_0,\lambda_{\level'}(\tol / \sigma_0)) = \level'$$ by definition of $\lambda_{\level'}(\tol / \sigma_0)$.
It then follows from the bijectivity of $\Marcum(\tol / \sigma_0,\mycdot)$ and the foregoing equalities that} $\lambda_{\level'}(\tol/ \sigma_0) = \| \ybm \| / \sigma_0$.
\pastorv{According to Lemma \ref{Lemma: tsub decreases}, we have:}
\begin{onecol}
$$
\begin{array}{lll}
\mybig \{ \level \in (0,1): \tsub(\tol / \sigma_0) < \| \ybm \| / \sigma_0 \mybig \} 
& = & \pastorv{\mybig \{ \level \in (0,1): \tsub(\tol / \sigma_0) < \lambda_{\level'}(\tol/\sigma_0) \mybig \}} \myvspace \\
& = & (\level',1) 
\end{array}
$$
\end{onecol}
\begin{twocol}
\begin{align}
\nonumber
\mybig \{ \level \in (0,1): &~\tsub(\tol / \sigma_0) < \| \ybm \| / \sigma_0 \mybig \} \\
\nonumber
& = \pastorv{\mybig \{ \level \in (0,1): \tsub(\tol / \sigma_0) < \lambda_{\level'}(\tol/\sigma_0) \mybig \}} \\
\nonumber
& = (\level',1) 
\end{align}
\end{twocol}
Therefore, $\level' = \inf \mybig \{ \level \in (0,1): \Tfrak_{\sigma_0 \tsub(\tol/\sigma_0)}(\ybm) = 1 \mybig \}$,
\pastorv{where $\Tfrak_{\sigma_0 \tsub(\tol/\sigma_0)}$ is defined according to (\ref{Eq:Thresholding test from above}) for all $\ybm \in \Rset^\dim$ by setting:
\begin{equation}
\Tfrak_{\sigma_0 \tsub(\tol/\sigma_0)}(\ybm) = \left \{
\begin{array}{lll}
0 & \hbox{ if } & \| \ybm \| \leqslant \sigma_0 \tsub(\tol/\sigma_0) \\
1 & \hbox{ if } & \| \ybm \| > \sigma_0 \tsub(\tol/\sigma_0).
\end{array}
\right.
\nonumber
\end{equation}
According to \cite{RDT}, $\Tfrak_{\sigma_0 \tsub(\tol/\sigma_0)}$ satisfies several optimality criteria for testing whether $\| \Ctr \| \leqslant \tol$ or not when we observe $\Obs \thicksim \Ncal(\Ctr, \Id)$. Therefore, 
\begin{equation}
\label{Eq:pvalue}
\pval{\ybm}{\sigma_0} \stackrel{\text{def}}{=} \Marcum(\tol / \sigma_0, \pastorv{\| \ybm \| / \sigma_0})
\end{equation}
can be regarded as the p-value of $\Tfrak_{\sigma_0 \tsub(\tol/\sigma_0)}$ for testing the hypothesis $\| \Ctr \| \leqslant \tol$.} \pastorv{The Wald test of Section \ref{subsec: pvalue} for testing $\Ctr = 0$ or not, when the observation is $\Obs \thicksim \Ncal(\Ctr, \Id)$, then corresponds to the particular case $\tol = 0$, for which we have  $\tsub(\tol) = \TheThreshold{\sigma_0}$ and $\pval{\ybm}{\sigma_0} = \Marcum(0, \| \ybm \| / \sigma_0)$.}
%
%


\section{Robustness of the M-estimation function}
\label{App: robustness}

In \algo~and \dalgo, centroids are calculated by M-estimation since, given a cluster, data from other clusters can be regarded as outliers. Our claim is then that the \pastorv{weight function} $w$, specified by \eqref{Eq: our weight}, is particularly suitable because it is the p-value associated with an optimal test, \pastorv{namely the Wald test}, aimed at deciding whether two observations lie within nearby clusters or not. As such, $w$ can be expected to be discriminating enough between data from different clusters. In this section, we thus analyze to what extent the M-estimator \pastorv{based on this weight function} is actually robust to outliers. This analysis can be carried out by studying the influence function of the estimator.

\pastorv{As announced at the beginning of this appendices, we carry out the computation in the more general case where the weight function is $w = \Marcum( \tol, \bullet)$, which requires handling $\tsub(\tol)$ instead of merely $\TheThreshold{}$. The presence of $\tol$ does not complexify the analysis and opens more general prospects.}


Following \cite[Definition 5, p. 230]{Hampel86}, the general case of an $M$-estimator of some parameter $\Ctr \in \Theta \subset \Rset^\dim$ given cumulative distribution function (cdf) $\Fbb$, is the solution $\Tbm(\Fbb)$ in $\tbm$ to the equation:
\begin{equation}
\label{Eq: integral equation for M estimation}
\displaystyle \int \Psibm(\ybm,\tbm) \d \Fbb(\ybm) = 0
\end{equation}
with $\Psibm: \Rset^\dim \times \Theta \rightarrow \Rset^\dim$. 
In particular, given $d$-dimensional vectors 
$\Obs_1, \ldots, \Obs_N \stackrel{\text{iid}}{\thicksim} \Gbb$ where $\Gbb$ is a given cdf, we can consider the empirical probability distribution 
$\Gbb_N^* = \frac{1}{N} \sum_{n=1}^N \delta_{\Obs_n},$ 
which puts mass $1/N$ at each $\Obs_1, \ldots, \Obs_N$. The solution $\Tbm(\Gbb_n^*)$ in $\tbm$ to Eq. (\ref{Eq: integral equation for M estimation}) with $\Fbb = \Gbb_n^*$ is the standard $M$-estimator for the sample $\Obs_1, \cdots, \Obs_N$ of distribution $\Gbb$.

Now, choose $\Psibm(\ybm,\tbm) = \Psi(\ybm - \tbm)$, where $\Psi$ is given by Eq. (\ref{Eq: Our Psi}). Set $\tbm = (t_1, \ldots, t_\dim)^\transpose$, $\ybm = (y_1, \ldots, y_\dim)^\transpose$ and $\Psi = (\Psi_1, \ldots, \Psi_\dim)$. For the general $M$-estimator $\Tbm(\Fbb)$ obtained by solving Eq. (\ref{Eq: integral equation for M estimation}), the differentiability of $\Psi$ induces that we can define the $\dim \times \dim$ matrix
$\Bbm = (B_{j,k})_{1 \leqslant j, k \leqslant \dim}$ with:
\begin{equation}
 B_{j,k} = - \displaystyle \int \left[ \left. \frac{\partial \Psi_j(\obs,\tbm)}{\partial t_k}\right|_{\tbm = \Tbm(\Fbb)}\right] \d \Fbb(\ybm)
\end{equation}
Since we have:
\begin{onecol}
\begin{equation}
\nonumber
\frac{\partial \Psi_j(\obs,\tbm)}{\partial t_k}
= 
\left \{
\begin{array}{lll}
- \, \dfrac{(y_j - t_j)(y_k - t_k)}{\| \obs - \tbm \|} Q'_{d/2}(\tol, \| \obs - \tbm \|), & \text{if} & j \neq k \myvspace \\
- \, Q_{d/2}(\tol, \| \obs - \tbm \|) - \dfrac{(y_k - t_k)^2}{\| \obs - \tbm \|} Q'_{d/2}(\tol, \| \obs - \tbm \|) , & \text{if} & j = k
\end{array}
\right.
\end{equation}
\end{onecol}
\begin{twocol}
\small
\begin{align}
\nonumber
& \frac{\partial \Psi_j(\obs,\tbm)}{\partial t_k} = \\
\nonumber
& 
\left \{
\begin{array}{lll}
- \dfrac{(y_j - t_j)(y_k - t_k)}{\| \obs - \tbm \|} Q'_{d/2}(\tol, \| \obs - \tbm \|), & \text{if $j \neq k$} \\
- Q_{d/2}(\tol, \| \obs - \tbm \|) - \dfrac{(y_k - t_k)^2}{\| \obs - \tbm \|} Q'_{d/2}(\tol, \| \obs - \tbm \|) , & \text{if $j = k$}
\end{array}
\right.
\end{align}
\normalsize
\end{twocol}
the matrix $\Bbm$ can be expressed as:
\begin{onecol}
\begin{equation}\label{eq:Bderivative}
\Bbm 
 = \displaystyle \int Q_{d/2}(\tol, \| \obs - \Tbm(\Fbb) \|) \d \Fbb(\ybm) \Id + \displaystyle \int Q'_{d/2}(\tol, \| \obs - \Tbm(\Fbb) \|) (\ybm - \Tbm(\Fbb))(\ybm - \Tbm(\Fbb))^\transpose \d \Fbb(\ybm)
\nonumber
\end{equation}
\end{onecol}
\begin{twocol}
\small
\begin{align}
\nonumber
\Bbm = \displaystyle \int & Q_{d/2}(\tol, \| \obs - \Tbm(\Fbb) \|) \d \Fbb(\ybm) \Id \\
\nonumber
& + \displaystyle \int Q'_{d/2}(\tol, \| \obs - \Tbm(\Fbb) \|) (\ybm - \Tbm(\Fbb))(\ybm - \Tbm(\Fbb))^\transpose \d \Fbb(\ybm)
\nonumber
\end{align}
\normalsize
\end{twocol}
The first integral in the right hand side (rhs) of the second equality above is positive. Therefore, the eigenvalues of the symmetric matrix $\Bbm$ are all positive and $\Bbm$ is invertible. According to \cite[Eq. (4.2.9), p. 230]{Hampel86}, the invertibility of $\Bbm$ \pastorv{makes it possible to calculate the} influence function as:
\begin{align}\notag
  \IfF(\obs;T,F) 
& = \Bbm^{-1}\Psibm(\obs,\Tbm(\Fbb)) \myvspace \\ 
& = \Bbm^{-1} (\obs - \Ctr) Q_{d/2}(\tol, \| \obs - \Ctr \|)
\nonumber
 \end{align}
Since the Marcum function is continuous, the influence function $\IfF(\obs;T,F)$ is also continuous. The norm of the influence function can now be bounded by:
\begin{equation}
 \| \IfF(\obs;T,F) \| \leq ||| \Bbm^{-1} ||| \times  \|\obs - \Ctr \| \times Q_{d/2}(\tol, \| \obs - \Ctr \|)
 \nonumber
\end{equation}
where $||| \cdot |||$ refers to the norm of a matrix. Let us consider the function $f : t \in [0,+\infty[ \rightarrow t Q_{d/2}(\tol, t)$. This function is continuous. If $\tau \neq 0$, then $Q_{\dim/2}(\tau,t) \thicksim \left ( {t}/{\tau} \right )^{(\dim-1)/2} Q(t-\tau)$ from \cite[p. 1167, Eq. (4)]{Sun2010}, where $Q$ is the Gaussian $Q$ -- function. It follows from \cite[Corollary 1]{Chang2011} that $Q(x) \leqslant \frac{1}{2}e^{-x^2/2}$ and thus, that $\lim_{t \rightarrow \infty} f(t) = 0$. We have the same result if $\tau = 0$. \pastorv{Indeed}, if $\tau=0$, it follows from \cite[p. 1168, Eq. (11)]{Sun2010} and a straightforward change of variable that: 
\small
$$
Q_{\dim/2}(0,x) 
= \dfrac{1}{2^{\dim/2} \Gamma(\dim/2)} \displaystyle \int_{x^2}^\infty t^{\dim/2-1} e^{-t/2} \dt 
\leqslant e^{-x^2/4} 2^{\dim/2}
$$
\normalsize
Since $f(0) = 0$ and $f$ is continuous, we derive from the foregoing that $f$ is upper-bounded and so is the norm of the influence function. Since the influence function is continuous and bounded, the estimator gross error sensitivity 
$\gamma^{\star} = \sup_{\obs} \| \IfF(\obs;T,F) \|$
is finite. This shows that our estimator is robust to outliers.


\section{Asymptotic behaviors}


\label{App: asymptotic behaviors}
\begin{Lemma}
\label{Lemma: asymptotic behaviors}
\pastorv{If $\VZ{\xibm} \thicksim \Ncal(\xibm,\Id)$} with $\xibm \in \Rset^\dim$, then: 
\medskip \\
\noindent
(i) $\Expect{w ( \| \VZ{0} \|^2 ) \VZ{0}} = 0$ \medskip \\
\noindent
(ii) $\lim\limits_{\| \xibm \| \rightarrow \infty} \Expect{w ( \| \VZ{\xibm} \|^2 )} = 0$ \medskip \\
\noindent
(iii) $\displaystyle \lim\limits_{\| \xibm \| \rightarrow \infty} \Expect{w ( \| \VZ{\xibm} \|^2 ) \VZ{\xibm}}
= 0$
\end{Lemma}

\begin{IEEEproof} As a preliminary result, we recall that:
\begin{equation}
\label{Eq:Prop of Q}
\lim\limits_{t \to \infty} \myQ{t} = 0, 
\end{equation}
which derives from the fact that $\myQ{t} = 1 - \Fbb_{\chi^2(\tau^2)}(t^2)$ \pastorv{\cite[Eq. (8)]{Sun2010}}, where $\Fbb_{\chi_2(\tau^2)}$ is the non-centered $\chi_2$ distribution with $\dim$ degrees of freedom and non-centrality parameter $\tau^2$.
\\
\noindent
{\em Proof of statement (i):} The vector $\Expect{w ( \| \VZ{0} \|^2 ) \VZ{0}}$ is $\dim$-dimensional. Its first component is:
$$
\Expect{w ( \| \VZ{0} \|^2 ) Z_1(0)}
=
\displaystyle \int_{\Rset^\dim} w \Big ( \sum_{k=1}^\dim z_k^2 \Big ) z_1 \varphi(\zbm) \d \zbm
$$
where $Z_1(0)$ is the first component of $\VZ{0}$, $\zbm = (z_1, z_2, \ldots, z_\dim)^\transpose$ and $\varphi$ is the probability density function (pdf) of the standard distribution $\Ncal(0,\Id)$. Since $w \Big ( \sum_{k=1}^\dim z_k^2 \Big )$ is bounded by $1$ and $\Expect{| Z_1(0) |}$ is finite, Fubini's theorem applies and we have:
\begin{onecol}
$$
\begin{array}{lll}
\Expect{w ( \| \VZ{0} \|^2 ) Z_1(0)}
= \\
\dfrac{1}{(2 \pi)^{N/2}} \displaystyle \int \left ( \ldots \left ( \displaystyle \int \left ( \int w \Big ( \sum_{i=1}^\dim z_i^2 \Big ) z_1 e^{-{z_1^2}/{2}} \d z_1 \right ) e^{-{z_2^2}/{2}} \d z_2  \right ) \ldots \right ) e^{-{z_\dim^2}/{2}} \d z_\dim.
\end{array}
$$
\end{onecol}
\begin{twocol}
\small
\begin{align}
\nonumber
& \Expect{w ( \| \VZ{0} \|^2 ) Z_1(0)} = \dfrac{1}{(2 \pi)^{N/2}} \\
\nonumber
& \displaystyle \int \! \left ( \ldots \! \left ( \displaystyle \int \! \left ( \int w \Big ( \sum_{i=1}^\dim z_i^2 \Big ) z_1 e^{-\frac{z_1^2}{2}} \d z_1 \right ) \!  e^{-\frac{z_2^2}{2}} \d z_2  \right ) \! \ldots \right ) \! e^{-\frac{z_\dim^2}{2}} \d z_\dim.
\end{align}
\normalsize
\end{twocol}
Since $\displaystyle \int \! w \Big ( \sum_{i=1}^\dim z_i^2 \Big ) z_1 e^{-{{z_1^2}/{2}}} \d z_1 = 0$ because the integrand is odd, it follows that $\Expect{w ( \| \VZ{0} \|^2 ) Z_1(0)}  = 0.$ The same type of computation holds for any component of $\Expect{w ( \| \VZ{0} \|^2 ) \VZ{0}}$. Thence the result. 
\medskip \\
\noindent
{\em Proof of statement (ii):} By definition of $w$ and $\VZ{\xibm}$, we have:
\small
\begin{equation}
\label{Eq: expressions w}
w ( \| \VZ{\xibm} \|^2 ) = \myQ{\| \VZ{\xibm} \|} = \myQ{\| \xibm + \Noise \|}
\end{equation}
\normalsize
Since $\| \xibm + \Noise \| \geqslant \| \xibm \| - \| \Noise \|$ and $\myQ{\mycdot}$ is decreasing, we derive from \eqref{Eq:Prop of Q} and \eqref{Eq: expressions w} that:
\begin{equation}
\label{Eq: useful cvgce}
\lim\limits_{\| \xibm \| \rightarrow \infty} w ( \| \VZ{\xibm} \|^2 ) = 0 \quad \text{(a-s)}
\end{equation}
The result then derives from Lebesgue's dominated convergence theorem.
\medskip \\
\noindent
{\em Proof of statement (iii):} We begin by writing that 
\begin{onecol}
\begin{equation}
\label{Eq: Basic equality for (iv)}
\Expect{w ( \| \VZ{\xibm} \|^2 ) \VZ{\xibm}} = \Expect{w ( \| \VZ{\xibm} \|^2 )} \xibm + \Expect{w ( \| \VZ{\xibm} \|^2 ) \Noise} .
\end{equation}
\end{onecol}
\begin{twocol}
\begin{align}
\label{Eq: Basic equality for (iv)}
\Ebb \big [ w ( \| & \VZ{ \xibm} \|^2 ) \VZ{\xibm} \big ] \nonumber \\ 
& = \Expect{w ( \| \VZ{\xibm} \|^2 )} \xibm + \Expect{w ( \| \VZ{\xibm} \|^2 ) \Noise} .
\end{align}
\end{twocol}
Set $\xibm = ( \xi_1, \xi_2, \ldots, \xi_\dim )^\transpose$. The first component of the first term to the rhs of the equality above can be rewritten as:
\begin{equation}
\label{Eq: inequalities for (iv)}
\Expect{w ( \| \VZ{\xibm} \|^2 )} \xi_1 = \dfrac{1}{(2 \pi)^{\dim/2}} \displaystyle \int_{\Rset^\dim} w \left ( \| \ybm \|^2 \right ) \xi_1 e^{-\frac{1}{2} \| \ybm - \xibm \|^2} \d \ybm
\end{equation}
The inequality $-\| \xibm \| \leqslant \xi_1 \leqslant \| \xibm \|$ induces that:
$$
-\| \xibm \| e^{-\frac{1}{2} \| \ybm - \xibm \|^2} \leqslant \xi_1 e^{-\frac{1}{2} \| \ybm - \xibm \|^2} \leqslant \| \xibm \| e^{-\frac{1}{2} \| \ybm - \xibm \|^2}
$$
For any given $\ybm \in \Rset^\dim$, the left and right bounds in the inequality above tend to $0$ when $\| \xibm \|$ tends to $\infty$. Lebesgue's dominated convergence theorem applied to \eqref{Eq: inequalities for (iv)} yields that $$\lim\limits_{\| \xibm \| \to \infty} \Expect{w ( \| \VZ{\xibm} \|^2 )} \xi_1 = 0$$ The same reasoning holds for any component of $\xibm$. Therefore,
\begin{equation}
\label{Eq: lim of expect1}
\lim\limits_{\| \xibm \| \to \infty} \Expect{w ( \| \VZ{\xibm} \|^2 )} \xibm = 0
\end{equation}
As far as the second term to the rhs of \eqref{Eq: Basic equality for (iv)} is concerned, we have ${w ( \| \VZ{\xibm} \|^2 ) \Noise} \leqslant \| \Noise \|$. Since $\Expect{\| \Noise \|} < \infty$, we derive from Lebesgue's dominated convergence theorem and \eqref{Eq: useful cvgce} that: 
\begin{equation}
\label{Eq: lim of expect2}
\lim\limits_{\| \xibm \| \to \infty} \Expect{w ( \| \VZ{\xibm} \|^2 ) \Noise} = 0
\end{equation}
Thence the \pastorv{result} as a consequence of (\ref{Eq: Basic equality for (iv)}), (\ref{Eq: lim of expect1}) and (\ref{Eq: lim of expect2}). 
\end{IEEEproof}


\section{}\label{App:D}

\begin{onecol}
\begin{Lemma}
\label{Lemma: last useful lemma}
\pastorv{Let} $\Zbm$ be a $\dim$-dimensional Gaussian vector with covariance matrix $\Id$. \pastorv{If} $\fbm: \Rset^\dim \to [0,\infty)$ is \pastorv{non-null}, continuous \pastorv{and even in each coordinate of $\xbm = (x_1, \ldots, x_\dim)^\transpose \in \Rset^\dim$ so that $$\fbm(x_1, \ldots,x_{i-1},x_i,x_{i+1}, \ldots, x_\dim) = \fbm(x_1, \ldots,x_{i-1},-x_i,x_{i+1}, \ldots, x_\dim),$$}  then:
\[
\Expect{\fbm ( \Zbm ) \Zbm} = 0 \, \, \, \, \text{if and only if} \, \, \, \, \Expect{\Zbm} = 0.
\]
\end{Lemma}
\end{onecol}
\begin{twocol}
\begin{Lemma}
\label{Lemma: last useful lemma}
\pastorv{If} $\fbm: \Rset^\dim \to [0,\infty)$ is \pastorv{non-null}, continuous and even in each coordinate of $\xbm = (x_1, \ldots, x_\dim)^\transpose \in \Rset^\dim$ in the sense that $\fbm(x_1, \ldots,x_{i-1},x_i, x_{i+1}, \ldots, x_\dim) = \fbm(x_1, \ldots,x_{i-1},-x_i,x_{i+1}, \ldots, x_\dim)$ then, for any $\Zbm \thicksim \Ncal(\xibm,\Id)$. 
$\Expect{\fbm ( \Zbm ) \Zbm} = 0 \, \, \, \, \text{if and only if} \, \, \, \, \Expect{\Zbm} = 0.$
\end{Lemma}
\end{twocol}

\begin{IEEEproof}
If $\dim = 1$, $\Zbm$ is a random variable $Z \thicksim \Ncal(\xi,1)$ and $\fbm$ is a nonnegative real function $f: \Rset \to [0,\infty)$. We have $\Expect{f(Z)Z} \hspace{-0.25cm} = \hspace{-0.25cm} \frac{1}{\sqrt{2 \pi}} \int_{-\infty}^\infty f(z) z e^{-\frac{1}{2} (z-\xi)^2} \dz$. By splitting this integral in two, symmetrically with respect to the origin, and after the change of variable $t = -z$ in the integral from $-\infty$ to $0$ resulting from the splitting, some routine algebra leads to
$\Expect{f(Z)Z} = \frac{1}{\sqrt{2 \pi}} \int_{0}^\infty f(-t) t e^{-\frac{1}{2} (t^2 + \xi^2)} \left ( e^{\xi t} - e^{-\xi t} \right ) \dt$. The integrand in this integral is non-negative and continuous. Therefore, $\Expect{f(Z)Z} =0$ implies that $f(-t) t e^{-\frac{1}{2} (t^2 + \xi^2)} \left ( e^{\xi t} - e^{-\xi t} \right ) = 0$ for any $t$, which induces $\xi = 0$. \pastorv{The converse is straightforward.}

In the $\dim$-dimensional case, set $\Zbm = (Z_1, Z_2, \ldots, Z_\dim)$ and denote the expectation $\Expect{\Zbm}$ of $\Zbm$ by $\xibm = (\xi_1, \xi_2, \ldots, \xi_\dim)$. 
Let $\fbm: \Rset^\dim \to [0,\infty)$ be a \pastorv{nonnull} and continuous function. Clearly, $\Expect{f(\Zbm) \Zbm} = 0$ if and only if $\Expect{f(\Zbm) Z_i} = 0$ for each coordinate $Z_i$, $i \in \{1, 2, \ldots, \dim \}$. For the first coordinate $Z_1$ of $\Zbm$, it follows from Fubini's theorem that:
\[
\vspace{-0.1cm}
\Expect{\fbm ( \Zbm ) Z_1} =\displaystyle \int f(z_1) \, z_1 \, e^{-\frac{1}{2} (z_1 - \xi_1)^2} \dz_1
\vspace{-0.2cm}
\]
with 
$$\vspace{-0.1cm}
f(z_1) \! = \! \dfrac{1}{(2 \pi)^{d/2}} \! \displaystyle \int \! \fbm(z_1, \ldots, z_\dim) e^{-\frac{1}{2} \sum_{k=2}^\dim (z_k - \xi_k)^2} \dz_2  \ldots  \dz_\dim.$$ The function $f$ is defined on $\Rset$, continuous and non-negative too. The result then follows from the monodimensional case treated above.
\end{IEEEproof}

\section{Proof of  $\Expect{w \left ( \| \Noise \|^2 \right )^2 \Noise \Noise^\transpose} = \Expect{w \left ( \| \Noise \|^2 \right )^2} \Id$ when $\Noise \thicksim \Ncal(0,\Id)$}
\label{App: Correlation matrix}
\pastorv{Set $\Noise = (\Noise_1, \Noise_2, \ldots, \Noise_\dim)^\transpose$ and} compute the term $c_{i,j}$ located at the $i$th line and $j$th colum of the matrix $\Expect{w \left ( \| \Noise \|^2 \right )^2 \Noise \Noise^\transpose}$ with $i \ne j$. We have:
$$
c_{i,j} 
= \Expect{w \left ( \| \Noise \|^2 \right )^2 \Noise_i \Noise_j} 
= \displaystyle \int w \left ( \| \xbm \|^2 \right )^2 x_i x_j \varphi(\xbm) \d \xbm
$$
where $\varphi$ is the pdf of $\Noise \thicksim \Ncal(0,\Id)$. By independence of the components of $\Noise$ and Fubini's theorem, the foregoing implies: 
\begin{onecol}
$$
c_{i,j} 
= \dfrac{1}{(2\pi)^{\dim/2}} \displaystyle \int \prod_{k=1, k\ne i,j}^\dim e^{-x_k^2/2} \d x_k \displaystyle \int \left ( \, \displaystyle \int w \left ( \| \xbm \|^2 \right )^2 x_i \varphi(x_i) \d x_i \, \right ) x_j \varphi(x_j)  \d x_j
$$
\end{onecol}
\begin{twocol}
\small
\begin{align}
\nonumber
& c_{i,j} = \\ \nonumber
& \dfrac{1}{(2\pi)^{\frac{\dim}{2}}} \! \! \displaystyle \int \! \! \! \! \prod_{k=1, k\ne i,j}^\dim \! \! \! e^{-\frac{x_k^2}{2}} \d x_k \! \displaystyle \int \! \left ( \displaystyle \int \! w \left ( \| \xbm \|^2 \right )^2 x_i \varphi(x_i) \d x_i \! \right ) \! x_j \varphi(x_j) \d x_j
\end{align}
\normalsize
\end{twocol}
Since $\displaystyle \int w \left ( \| \Noise \|^2 \right )^2 x_i \varphi(x_i) \d x_i = 0$ because the integrand in this integral is odd, we conclude that $c_{i,j} = 0$ for $i \ne j$.}

We now compute $c_{i,i}$ for $i = 1, \ldots, \dim$. Similarly to above, Fubini's theorem implies that:
\small
\begin{align}
\begin{array}{lll}
c_{i,i} & = \Expect{w \left ( \| \Noise \|^2 \right )^2 \Noise^2_i} 
\\
& = \dfrac{1}{(2\pi)^{\dim/2}} \! \displaystyle \int \left ( \displaystyle \int w \left ( \| \xbm \|^2 \right )^2 x_i^2 e^{-\frac{x_i^2}{2}} \d x_i \right ) \! \! \! \prod_{k=1, k\ne i}^\dim \! \! e^{-\frac{x_k^2}{2}} \d x_k  
\\
& = \dfrac{1}{(2\pi)^{\dim/2}} \! \displaystyle \int \Ebb \Big [ w \big ( \Noise_i^2 + \! \displaystyle \sum_{j=1,j\ne i}^\dim x_j^2 \big )^{2} \, \Big ]  \prod_{k=1, k\ne i}^\dim \! e^{-\frac{x_k^2}{2}} \d x_k  
\\ \notag
& = \Expect{w \left ( \| \Noise \|^2\right )^2}
\end{array}
\end{align}
\normalsize
the last equality being obtained by iterating the integration over all the components of $\Noise$.

\bibliographystyle{IEEEtran}
\bibliography{IEEEabrv,Refs}

\end{document}